\documentclass[11pt,a4paper]{amsart}

\usepackage{graphicx}
\usepackage{pstricks}
\usepackage{amsmath}
\usepackage{amsfonts}
\usepackage{amssymb}
\usepackage{amsthm}
\usepackage{url}
\usepackage{mathdots}
\usepackage{nicefrac}
\usepackage{subfigure}
\input xy
\xyoption{all}
\usepackage{pst-all}

\newtheorem{theorem}{Theorem}
\newtheorem*{theorem*}{Theorem}

\newtheorem{definition}[theorem]{Definition}
\newtheorem{lemma}[theorem]{Lemma}

\newtheorem{proposition}[theorem]{Proposition}
\theoremstyle{definition}
\newtheorem{example}[theorem]{Example}
\newtheorem{remark}[theorem]{Remark}
\newtheorem*{remark*}{Remark}

\begin{document}

\author{J.J. S\'anchez-Gabites}
\title[A dynamical interpretation of the connection map]{A dynamical interpretation of the connection map of an attractor-repeller decomposition}

\address{Facultad de Ciencias Matem{\'a}ticas. Universidad Complutense de Madrid. 28040 Madrid (Espa{\~{n}}a)}
\email{jajsanch@ucm.es}

\subjclass[2020]{37B30, 37B35, 55N99}
\keywords{Conley index, attractor-repeller decomposition, connection homomorphism}
\thanks{The author is supported by the Spanish Ministerio de Ciencia e Innovaci\'on through grants PID2021-126124NB-I00 and RYC2018-025843-I}

\begin{abstract} In Conley index theory one may study an invariant set $S$ by decomposing it into an attractor $A$, a repeller $R$, and the orbits connecting the two. The Conley indices of $S$, $A$ and $R$ fit into an exact sequence where a certain connection homomorphism $\Gamma$ plays an important role. In this paper we provide a dynamical interpretation of this map. Roughly, $R$ ``emits'' an element of its Conley index as a ``wavefront'', part of which intersects the connecting orbits in $S$. This subset of the wavefront evolves towards $A$ and is then ``received'' by it to produce an element in its Conley index.
\end{abstract}

\maketitle

\section{Introduction}

Consider a continuous or discrete dynamical system on a manifold. A compact invariant set $S$ is isolated in the sense of Conley if it is the largest invariant subset of some small neighbourhood of itself. Any isolated invariant set $S$ has a well defined homological Conley index: this is a (graded) vector space $CH_*(S)$ which, roughly speaking, captures the homology of the local unstable set of $S$. In this sense the Conley index provides a partial description of the local dynamics in a neighbourhood of $S$. When the dynamics is discrete $CH_*(S)$ is endowed with a distinguished isomorphism, but this will play no role for the moment. 

An attractor-repeller decomposition of a compact invariant set $S$ consists of two disjoint compact invariant sets $A, R \subseteq S$ such that the orbit of every point in $S \setminus (A \cup R)$ approaches $R$ in the past and $A$ in the future. These are called the connecting orbits (from $R$ to $A$). When $S$ is isolated the sets $A$ and $R$ are also isolated so they all have well defined Conley indices, and it can be shown that they all fit into an exact sequence \[\xymatrix{\ldots \ar[r] & CH_q(A) \ar[r] & CH_q(S) \ar[r] & CH_q(R) \ar[r]^-{\Gamma_q} & CH_{q-1}(A) \ar[r] & \ldots}\] called the exact sequence of the attractor-repeller decomposition. This leads immediately to the equation $\dim CH_q(A) + \dim CH_q(R) = \dim CH_q(S) + (\dim {\rm im}\ \Gamma_q + \dim {\rm im}\ \Gamma_{q+1})$, and inductively to the so-called Morse equations of a Morse decomposition, one of the cornerstones of Conley index theory (see \cite{conleyzehnder1} for the case of flows and \cite{mro1} for the discrete case).

In order to extract dynamical information from the above sequence (or, by extension, the Morse equations of a Morse decomposition) it would be desirable to have a dynamical interpretation of the map $\Gamma$, called the connection homomorphism. More specifically, one would like to understand what information about the set of connecting orbits it captures. It is well known (and easy to prove) that when there are no connecting orbits between $R$ and $A$ (i.e. $S = A \cup R$) then $\Gamma = 0$. Beyond this, and even when the structure of the set of connecting orbits is very simple, the computation of $\Gamma$ is nontrivial. In this paper we will obtain an expression of $\Gamma$ as a composition of three maps each of which has a (hopefully) clear intuitive meaning. This provides a dynamical interpretation of the map $\Gamma$ and a convenient way to compute it.

We shall begin by describing two processes that are associated to any isolated invariant set and have nothing to do with attractor-repeller decompositions: an isolated invariant set can ``emit'' elements from its Conley index along its unstable manifold and can also ``receive'' chains in phase space lying in an infinitesimal neighbourhood of its stable manifold to produce elements in its Conley index. These two processes will be encoded by (a family of) maps denoted by $\partial$ and $\Omega$ respectively. Using these ideas, a crude description of $\Gamma$ goes as follows. An element in the Conley index $CH_*(R)$ of the repeller $R$ gets emitted along its unstable manifold via the map $\partial$; one may picture a sort of ``wavefront'' generated by $R$ which propagates through phase space carried by the dynamics. Part of the wavefront intersects the connecting orbits in $S$; it is thus a subset of the stable manifold of $A$. An infinitesimal neighbourhood of this part of the wavefront is received by $A$ via the map $\Omega$ to produce an element in $CH_*(A)$. We will show that the resulting map $CH_*(R) \longrightarrow CH_*(A)$ is the connection homomorphism $\Gamma$. More formally, Theorem \ref{teo:connecting} establishes that $\Gamma_q$ factors as a composition \[\xymatrix{CH_q(R) \ar[r]^-{\partial} & H_{q-1}(F) \ar[r]^-{\tau} & H_{q-1}(F,F \backslash S) \ar[r]^-{\Omega} & CH_{q-1}(A)}\] where one may think of $F$ as a set containing the wavefronts emitted by $R$. The map $\tau$ is the homomorphism induced in homology by the inclusion of $F$ in $(F,F \setminus S)$, so it trims wavefronts in $F$ down to an infinitesimal neighbourhood of $F \cap S$. Thus the equality $\Gamma = \Omega \circ \tau \circ \partial$ accords with the informal description of $\Gamma$ given above.

In the decomposition $\Gamma = \Omega \circ \tau \circ \partial$ the maps $\Omega$ and $\partial$ are completely unrelated to the fact that $A$ and $R$ are part of an attractor-repeller pair; they just capture the ``emitter'' role of $R$ and the ``receiver'' role of $A$. The map $\tau$, on the other hand, has nothing to do with dynamics and only records the geometric information of how the connecting orbits intersect the set of wavefronts $F$. This decoupling of $\Gamma$ is useful because it allows for computations with incomplete information (as is often the case), and we will illustrate this by obtaining simple proofs of some results of McCord \cite{mccord2} just by inspection. It also sheds some light on our initial question regarding what information about the set of connecting orbits is carried by $\Gamma$. This set appears in $\Gamma = \Omega \circ \tau \circ \partial$ only through the homology group $H_{q-1}(F,F \setminus S)$; one can immediately write $\dim {\rm im}\ \Gamma_q \leq \dim H_{q-1}(F, F \setminus S)$. Focus on the set $F \cap S$, which is roughly a section of the set of connecting orbits in $S$. In practice $F$ is often a manifold (perhaps slightly thickened), say of dimension $d-1$, and then by Alexander duality $\dim H_{q-1}(F, F \setminus S) = \dim H^{d-q}(F \cap S)$ so in particular $\dim H^{d-q}(F \cap S) \geq \dim {\rm im}\ \Gamma_q$. Thus $\Gamma_q$ provides a lower bound on the geometric complexity of $F \cap S$ and therefore also on the structure of the set of connecting orbits.

We will work with discrete dynamics generated by a continuous map, not necessarily a homeomorphism. Everything is equally valid (and simpler) for flows or semiflows and the adaptation is straightforward. The definition of both $\partial$ and $\Omega$ involves a sort of $\omega$-limit process which, though dynamically clear, requires a certain amount of algebra for its formalization. To a lesser extent, so does the definition of the discrete Conley index. This makes the whole paper slightly technical but should not obscure the dynamical interpretation. Section \ref{sec:algebraprelim} introduces the necessary algebraic machinery. In Section \ref{sec:conley} we recall the definition of the discrete Conley index. The maps $\partial$ and $\Omega$ are defined in Sections \ref{sec:emitters} and \ref{sec:receivers} respectively, giving first an example to motivate their definition and bridge the gap between dynamical intuition and algebraic formalism. Finally, in Section \ref{sec:ARdec} we obtain the decomposition of $\Gamma$ announced above.

\section{Preliminaries from linear algebra} \label{sec:algebraprelim}

Throughout this whole section $V$ is a vector space over some field $\mathbb{K}$ and $\phi : V \longrightarrow V$ is an endomorphism. At the end we shall discuss briefly the case when $V$ is an Abelian group.

\subsection{} For any polynomial $P(x) = \sum_{j=0}^d a_j x^j \in \mathbb{K}[x]$ we may construct the endomorphism $P(\phi) := \sum_{j=0}^d a_j \phi^j$ of $V$. If $v$ is an element in $V$, we abbreviate $P(\phi)(v)$ by $P \cdot v$. With this notation $\phi(v)$ is simply $x \cdot v$. The following two properties are straightforward to check: $(P+Q) \cdot v = P \cdot v + Q \cdot v$ and $P \cdot (Q \cdot v) = (PQ) \cdot v$.

We say that an element $v \in V$ is \emph{algebraic} if it is a ``root of a polynomial''; i.e. if there exists a polynomial $P \in \mathbb{K}[x]$ with independent term equal to $1$ (we then say that $P$ is normalized) and such that $P \cdot v = 0$. For example, if $V$ is finite dimensional and $\phi$ is an isomorphism, every $v \in V$ is algebraic by the Cayley-Hamilton theorem.

\begin{proposition} \label{prop:alg} Let $AV$ be the set of algebraic elements of $V$. Then:
\begin{itemize}
	\item[(i)] $AV$ is a subspace of $V$.
	\item[(ii)] If $Q$ is any polynomial and $v \in AV$ then $Q \cdot v \in AV$.
	\item[(iii)] $\phi$ restricts to an isomorphism $\phi|_{AV} : AV \longrightarrow AV$.
\end{itemize}
\end{proposition}
\begin{proof} (i) If $v,w$ are annihilated by normalized polynomials $P$ and $Q$ then the normalized polynomial $PQ$ annihilates any linear combination of $v$ and $w$.

(ii) If the normalized polynomial $P$ annihilates $v$ it also annihilates $Q \cdot v$ since $P \cdot (Q \cdot v) = (PQ) \cdot v = Q \cdot (P \cdot v) = 0$.

(iii) Notice that $\phi(v) = x \cdot v$, so $\phi(AV) \subseteq AV$ by part (ii) applied to the polynomial $Q = x$. Now let $v \in AV$. Pick a normalized polynomial $P$ such that $P \cdot v = 0$. Writing $P = 1 -xQ$ the condition $P \cdot v= 0$ becomes $v = xQ \cdot v$. This implies that $v = \phi(Q \cdot v)$. Since $Q \cdot v$ is algebraic by part (ii), we see that $\phi|_{AV} : AV \longrightarrow AV$ is indeed surjective. To check that $\phi|_{AV}$ is injective assume $\phi(v) = 0$. This means $x \cdot v = 0$, and multiplying by $Q$ we have $0 = Q \cdot (x \cdot v) = (xQ) \cdot v = v$.
\end{proof}

In this paper it will always be the case that every element in $V$ is eventually algebraic, in the sense that for each $v \in V$ there exists $k$ (depending on $v$) such that $\phi^k(v) \in AV$. For example, if $\dim V = d$ is finite then this holds true: writing the characteristic polynomial of $\phi$ as $P = x^d + \ldots + a_k x^k$ with $a_k \neq 0$, it can be factored as $P =a_k x^k P'$ where $P'$ is normalized, and then $P \cdot v = 0$ implies that $P' \cdot (\phi^k(v)) = 0$ for every $v \in V$, so that $\phi^k(v)$ is algebraic. Notice that this also gives a method to compute $AV$: since $\phi^k(V) \subseteq AV$ and $\phi|_{AV}$ is an isomorphism of $AV$ we have $AV = \phi^d(AV) \subseteq \phi^d(V) \subseteq \phi^k(V) \subseteq AV$ which shows that $AV = {\rm im}\ \phi^d$.

Assuming that every element in $V$ is eventually algebraic we define a map $\omega : V \longrightarrow AV$ as follows: for $v \in V$, pick $k$ such that $\phi^k(v)\in AV$ and set $\omega(v) := (\phi|_{AV})^{-k} \circ \phi^k(v)$. \label{pg:can_proj} This is easily checked to be independent of $k$ since once $\phi^k(v)$ belongs to $AV$, further iteration under $\phi$ is the same as further iteration under $\phi|_{AV}$ and this cancels out with the $(\phi|_{AV})^{-k}$ term. Thus $\omega$ is a well defined linear map $V \longrightarrow AV$. Clearly $\omega|_{AV}$ is the identity (choose $k = 0$) so $\omega \circ \omega = \omega$. Also, since $\phi|_{AV}$ is an isomorphism, $\omega(v) = 0$ if and only if $\phi^k(v) = 0$ for large enough $k$. Recall that the set of $v \in V$ such that $\phi^k(v) = 0$ for some $k \geq 0$ is called the generalized kernel of $\phi$ and denoted by ${\rm gker}\ \phi$. Thus $\omega$ is a projection of $V$ onto $AV$ along the direction ${\rm gker}\ \phi$; in particular, there is a direct sum decomposition \[V = {\rm gker}\ \phi \oplus AV,\] which we call the canonical decomposition of $V$. The projection $\omega$ is called the canonical projection.

\subsection{} \label{subsec:shift} Let $V$ and $V'$ be vector spaces equipped with endomorphisms $\phi$ and $\phi'$. A map $\rho : V \longrightarrow V'$ is called equivariant if it commutes with the dynamics; i.e. $\rho \circ \phi = \phi' \circ \rho$. An equivariant map carries algebraic elements to algebraic elements: if $v \in V$ is annihilated by the normalized polynomial $P$, so is its image $\rho(v)$. A shift equivalence between $(V,\phi)$ and $(V',\phi')$ is a pair of equivariant maps \[\xymatrix{V \ar@<.5ex>[r]^{\rho} & V' \ar@<.5ex>[l]^{\sigma}}\] such that $\sigma \circ \rho = \phi^m$ and $\rho \circ \sigma = (\phi')^m$ for some $m \geq 0$ (see \cite[Section 4, pp. 3310ff.]{franksricheson1}). We will often abuse terminology and say that $\rho$ is a shift equivalence between $V$ and $V'$, omitting $\sigma$ and the maps $\phi$ and $\phi'$. 
The basic example of a shift equivalence is as follows: if $\dim V = d < +\infty$ we saw above that $\phi^d(v) \in AV$ for every $v \in V$; then $\rho := \phi^d : V \longrightarrow AV$ and $\sigma : AV \subseteq V$ define a shift equivalence between $V$ and $AV$.

The following proposition shows that shift equivalences preserves the notions defined above:

\begin{proposition} \label{prop:inv_shift} Let $\rho,\sigma$ be a shift equivalence from $(V,\phi)$ to $(V',\phi')$. Then:
\begin{itemize}
    \item[(i)] $\rho$ carries $AV$ isomorphically onto $AV'$.
    \item[(ii)] If $V$ admits a canonical decomposition, so does $V'$.
\end{itemize}
\end{proposition}
\begin{proof} (i) Both $\rho$ and $\sigma$ carry algebraic elements to algebraic elements since they are equivariant. Therefore from $\sigma \circ \rho = \phi^m$ we have $\sigma|_{AV'} \circ \rho|_{AV} = \phi|_{AV}^m$. Since $\phi|_{AV}$ is an isomorphism, $\rho|_{AV}$ is injective and $\sigma|_{AV'}$ is surjective. Looking at the composition in the reverse order we have that $\rho|_{AV}$ is surjective, so it is an isomorphism. In fact its inverse is easily checked to be  $\phi|_{AV}^{-m} \circ \sigma$.

(ii) We need to show that every $v' \in V'$ is eventually algebraic. By assumption $\sigma(v') \in V$ is eventually algebraic, so there exists $k$ such that $\phi^k(\sigma(v')) \in AV$. Then $\rho(\phi^k(\sigma(v'))) \in \rho(AV) = AV'$. Using $\rho \circ \phi^k \circ \sigma = \rho \circ \sigma \circ {\phi'}^k = {\phi'}^m \circ {\phi'}^k$, we finally get  ${\phi'}^{k+m}(v') \in  AV'$.
\end{proof}

\subsection{} Consider a sequence of vector spaces \[\xymatrix{V_1 \ar[r]^{\alpha_1} & V_2 \ar[r]^{\alpha_2} & V_3}\] where each $V_i$ is equipped with an endomorphism $\phi_i$ and the $\alpha_i$ are equivariant.

\begin{proposition} \label{prop:exact} Assume the above sequence is exact at $V_2$ and both $V_1$ and $V_3$ have canonical decompositions. Then the same is true of $V_2$. Moreover, the sequence of algebraic parts \[\xymatrix{AV_1 \ar[r]^{\alpha_1|} & AV_2 \ar[r]^{\alpha_2|} & AV_3}\] is exact at $AV_2$. (Here $\alpha_i|$ means the restriction of $\alpha_i$ to the algebraic part of its source space).
\end{proposition}
\begin{proof} Observe that the sequence of algebraic parts makes sense because the $\alpha_i$ are equivariant and therefore carry algebraic elements to algebraic elements. We prove that it is exact at $AV_2$. Clearly ${\rm im}\ \alpha_1| \subseteq \ker \alpha_2|$ since the same is true in the original exact sequence. To prove the other inclusion let $v_2$ be an algebraic element in $\ker \alpha_2$ and pick $v_1 \in V_1$ such that $\alpha_1(v_1) = v_2$. Decompose $v_1 = v_{1n} + v_{1a}$ canonically, so that $v_{1n} \in {\rm gker}\ \phi_1$ and $v_{1a}$ is algebraic. There exists $k \geq 0$ such that $x^k \cdot v_{1n} = 0$, and then $x^k \cdot v_1 = x^k \cdot v_{1a}$ so \[x^k \cdot \alpha_1(v_{1a}) = \alpha_1(x^k \cdot v_{1a}) = \alpha_1(x^k \cdot v_1) = x^k \cdot \alpha_1(v_1) = x^k \cdot v_2.\] Unravelling the notation, $\phi_2^k(\alpha_1(v_{1a})) = \phi_2^k(v_2)$. Now, both $\alpha_1(v_{1a})$ and $v_2$ belong to $AV_2$ and, since $\phi_2|_{AV_2}$ is an isomorphism, $\alpha_1(v_{1a}) = v_2$. Thus every algebraic element in $\ker \alpha_2$ is indeed the image of an algebraic element of $V_1$ via $\alpha_1$.

Now we prove that every element in $V_2$ is eventually algebraic and so $V_2$ has a canonical decomposition. Pick $v_2$ in $V_2$. Since $V_3$ has a canonical decomposition, there exist an integer $k_3 \geq 0$ and a normalized polynomial $P_3$ such that $x^{k_3}P_3 \cdot \alpha_2(v_2) = 0$. Using the equivariance of $\alpha_2$ we may rewrite this as $\alpha_2(x^{k_3}P_3 \cdot v_2) = 0$, and then exactness at $V_2$ ensures that there exists $v_1 \in V_1$ such that $\alpha_1(v_1) = x^{k_3}P_3 \cdot v_2$. Again since $V_1$ has a canonical decomposition, there exist an integer $k_1 \geq 0$ and a normalized polynomial $P_1$ such that $x^{k_1}P_1 \cdot v_1 = 0$. Acting with $\alpha_1$ on both sides of this yields $0 = x^{k_1}P_1 \cdot \alpha_1(v_1) = x^{k_1+k_3}P_1P_3 \cdot v_2$. Since $P_1 P_3$ is normalized, we have shown that $\phi_2^{k_1+k_2}(v_2)$ is algebraic. In other words, the algebraic part $AV_2$ of $V_2$ eventually absorbs any vector in $V_2$.
\end{proof}

\subsection{} The algebraic machinery developed in this Section works the same for Abelian groups instead of vector spaces. The only difference is that for Abelian groups it is no longer true that every finitely generated group $V$ admits a canonical decomposition: $\mathbb{Z}$ endowed with the map $\phi(v) = 2v$ is a simple example. Every element is annihilated by the polynomial $x-2$, but this is not normalized (nor can be normalized in $\mathbb{Z}[x]$).

\section{The discrete Conley index \label{sec:conley}}

The Conley index was originally defined in the context of flows \cite{conley1} and extended to discrete dynamics in several ways by different authors: Robbin and Salamon \cite{robinsalamon1}, Mrozek and Rybakowski \cite{mro2}, \cite{mroryb1}, Szymczak \cite{szymczak1}, and Franks and Richeson \cite{franksricheson1}.  In our setting all these definitions produce essentially the same object and we shall use that of Franks and Richeson for convenience, with slight changes in terminology.

\subsection{Background definitions} Consider a dynamical system generated by iterating a continuous map $f : M \longrightarrow M$ on some phase space $M$. A (full) orbit through a point $x \in M$ is a sequence $(x_n)_{n \in \mathbb{Z}}$ such that $x_0 = x$ and $x_{n+1} =f(x_n)$ for every $n$. Given a compact set $N \subseteq M$, its maximal invariant subset ${\rm Inv}(N)$ is the set of points in $N$ that belong to a full orbit entirely contained in $N$. A set $S \subseteq M$ is called an isolated invariante set if it has a compact neighbourhood $N$ such that $S = {\rm Inv}(N)$. The set $S$ is then compact and invariant in the sense that $f(S) = S$, although in general $S \subsetneq f^{-1}(S)$.

The stable manifold of a compact invariant set $S$ is the set $W^s(S)$ of points $x \in M$ whose positive semiorbit converges to $S$. More precisely, $x \in W^s(S)$ if and only if for every neighbourhood $U$ of $S$ in $M$ there exists $n_0 \in \mathbb{N}$ such that $(f^n(x))_{n \geq n_0} \subseteq U$. The unstable manifold $W^u(S)$ of $S$ is the dual notion but, since the dynamics is not injective, it has a slightly different definition: $x \in W^u(S)$ if and only if there exists a negative semiorbit $(x_n)_{n \leq 0}$ through $x$ such that for every neighbourhood $U$ of $S$ there exists $n_0$ with $(x_n)_{n \leq n_0} \subseteq U$. The terminology ``stable/unstable manifold'' is customary but in general these sets are not manifolds at all.

\subsection{The Conley index} Let $N$ be a closed subset of phase space. The immediate exit set of $N$ is the set of points $x \in N$ such that $f(x) \not\in {\rm int}\ N$. A filtration pair $(N,L)$ for an isolated invariant set $S$ is a compact pair $(N,L)$ such that both $N$ and $L$ are the closure of their interiors and:
\begin{itemize}
	\item[(i)] The maximal invariant subset of $\overline{N \backslash L}$ is precisely $S$.
	\item[(ii)] $L$ is a neighbourhood (in $N$) of the immediate exit set of $N$.
	\item[(iii)] $f(L) \cap \overline{N \backslash L} = \emptyset$.
\end{itemize}

Consider the pointed space $(N/L,[L])$. It can be shown that the map $f$ induces a continuous map $f_{N/L} : (N/L,[L]) \longrightarrow (N/L,[L])$ given by \[f_{N/L}([x]) := \left\{ \begin{array}{ll} [f(x)] & \text{ if } f(x) \in N \backslash L \\ {[L]} & \text{ otherwise} \end{array} \right.\] Thus we have a new dynamical system in $N/L$ which is essentially a localization of $f$ around $S$.

\begin{remark} When $L = \emptyset$ (which can happen only when $N$ is a trapping region and $S$ is the attractor inside it) the quotient space $N/L$ is defined to be $N$ union an isolated point $[L]$. This is done so that $N \setminus L$ is homeomorphic to $N/L \setminus \{[L]\}$ via the canonical projection for any $L$, even empty.
\end{remark}

Suppose that $(N,L)$ and $(N',L')$ are two different filtration pairs for the same set $S$. We will make extensive use of the fact that the pointed spaces $(N/L,[L])$ and $(N'/L',[L'])$ are related by a shift equivalence (\cite[Theorem 4.3, pp. 3311]{franksricheson1}). Explicitly, there exists a pair of equivariant continuous maps \[\xymatrix{(N/L,[L]) \ar@<.5ex>[r]^r & (N'/L',[L']) \ar@<.5ex>[l]^s}\] such that $s \circ r = f_{N/L}^m$ and $r \circ s = f_{N'/L'}^m$ for some integer $m \geq 0$. This notion of a shift equivalence is completely analogous to the one in Subsection \ref{subsec:shift}, only in the category of pointed compact spaces. We record here two particular forms of shift equivalences for later reference:

\begin{example} \label{ex:standard} The shift equivalences $r,s$ can always be chosen to be standard; i.e. have the form \[r([x]) = \left\{ \begin{array}{ll} [f^k(x)] & \text{ if } x \in N \backslash L \text{ and } f^k(x) \in N' \backslash L', \\ {[L']} & \text{ otherwise} \end{array} \right.\] and similarly for $s$ with $k$ replaced with $m-k$.
\end{example} 

\begin{example} \label{ex:inclusion} Assume that $N \subseteq N'$ and $L \subseteq L'$. Then the natural map $r : N/L \longrightarrow N'/L'$ given by $[x] \longmapsto [x]$ is a shift equivalence.
\end{example}
\begin{proof} Notice that there exist $k,\ell \geq 0$ such that:
\begin{itemize}
    \item [(i)] for every $x \in N'$, the trajectory segment $x,\ldots,f^k(x)$ visits $N$ or $L'$ (or both).
    \item[(ii)] for every $x \in N \cap L'$, the trajectory segment $x,\ldots,f^{\ell}(x)$ visits $L$.
\end{itemize}

Define $s : N'/L' \longrightarrow N/L$ as follows: \[s([x]) = \left\{ \begin{array}{ll} [f^{k+\ell}(x)] & \text{ if } x,\ldots,f^k(x) \in N' \setminus L' \text{ and } f^k(x),\ldots, f^{k+\ell}(x) \in N \setminus L, \\ {[L]} & \text{ otherwise.} \end{array} \right.\] 

Then $s$ is a shift inverse for $r$.
\end{proof}

Let $S$ be an isolated invariant set and $(N,L)$ a filtration pair for $S$. The Conley index of $S$ is defined as the shift equivalence class of $(N/L,[L])$ endowed with the map $f_{N/L}$. This depends only on $S$ but not on the filtration pair $(N,L)$. For computations one usually passes to homology: the homological Conley index of $S$ is the shift equivalence class (now in the category of vector spaces) of $H_*(N/L,[L])$ endowed with the endomorphism $(f_{N/L})_*$. Here $H_*$ means homology with coefficients in some field, typically $\mathbb{Q}$ or $\mathbb{C}$. We denote the homological Conley index of $S$ by $CH_*(S)$.

\subsection{} We introduce the following notation. For any filtration pair $(N,L)$ for $S$, we denote by $CH_*(N/L)$ the algebraic part of $H_*(N/L,[L])$ and by $\phi_{N/L}$ the restriction of $(f_{N/L})_*$ to $CH_*(N/L)$, which is an isomorphism by Proposition \ref{prop:alg}.

We make the following observations:

(i) Suppose $H_*(N/L,[L])$ is finite dimensional. Then by the discussion in Section \ref{sec:algebraprelim} it has a canonical decomposition $H_*(N/L,[L]) = {\rm gker}\ (f_{N/L})_* \oplus CH_*(N/L)$. Moreover, $(H_*(N/L,[L]),(f_{N/L})_*)$ is shift equivalent to $(CH_*(N/L),\phi_{N/L})$ and so $CH_*(N/L)$ endowed with $\phi_{N/L}$ is (a representative of) the homological Conley index of $S$.

(ii) Let $(N',L')$ be another filtration pair for $S$ and pick a shift equivalence $r : N/L \longrightarrow N'/L'$. Then $r_* : H_*(N/L,[L]) \longrightarrow H_*(N'/L',[L'])$ is a shift equivalence of vector spaces, and so by Proposition \ref{prop:inv_shift} it restricts to an isomorphism between $CH_*(N/L)$ and $CH_*(N'/L')$ which conjugates $\phi_{N/L}$ and $\phi_{N'/L'}$. Also, since $H_*(N/L,[L])$ has a canonical decomposition the same is true of $H_*(N'/L',[L'])$.

(iii) When phase space is Euclidean space, or more generally any differentiable or triangulable manifold, every isolated invariant set $S$ has a filtration pair $(N,L)$ such that $N$ and $L$ are manifolds or polyhedra respectively (\cite[Theorem 3.7, p. 3309]{franksricheson1}). In fact, for computations one works almost exclusively with this sort of filtration pairs. For any such pair the homology $H_*(N/L,[L])$ is indeed finite dimensional.

Summarizing, we have the following:

\begin{proposition} \label{prop:W2invariant} Let phase space be Euclidean space or, more generally, a differentiable or triangulable manifold. Let $S$ be an isolated invariant set and $(N,L)$ any filtration pair for $S$. Then the homology $H_*(N/L,[L])$ admits a canonical decomposition \[H_*(N/L,[L]) = {\rm gker}\ (f_{N/L})_* \oplus CH_*(N/L)\] and $CH_*(N/L)$ is (a representative of) the homological Conley index of $S$. Moreover, if $r$ is a shift equivalence between two filtration pairs $(N,L)$ and $(N',L')$, then $r_* : H_*(N/L,[L]) \longrightarrow H_*(N'/L',[L'])$ restricts to an isomorphism between $CH_*(N/L)$ and $CH_*(N'/L')$ which conjugates $\phi_{N/L}$ and $\phi_{N'/L'}$.
\end{proposition}

From a heuristic point of view one should consider the vector spaces $CH_*(N/L)$ and $CH_*(N'/L')$ as representing the same object (namely the homological Conley index $CH_*(S)$) only expressed in different ``coordinate systems'', with $r_*$ being a ``change of coordinates''. For this reason we will call $CH_*(N/L)$ the homological Conley index (of $S$) as well, in fact omitting the word ``homological'' in the sequel.

\begin{remark} It follows from the proposition that that $(CH_*(N/L),\phi_{N/L})$ is also isomorphic to the Leray reduction of $(f_{N/L})_*$, which is (a homological counterpart) of Mrozek's definition of the discrete Conley index \cite{mro2}. One can check that $CH_*(N/L)$ is also the \v{C}ech homology of the one-point compactification of the unstable manifold of $S$ endowed with the unstable topology, as in Robbin and Salamon \cite{robinsalamon1}. Thus in this setting all definitions of the Conley index are essentially the same.
\end{remark}

So far we have not specified what homology theory we use. Recall that all homology theories are naturally isomorphic over pairs of triangulable spaces (\cite[Section 10, p. 100ff.]{eilenbergsteenrod1}), so the argument leading to Proposition \ref{prop:W2invariant} is valid for every homology theory and shows that $CH_*(N/L)$ is independent of this choice. Since one of our goals is to obtain an intuitive interpretation of the connection homomorphism it would be best to use singular homology throughout. However, to define the map $\partial$ mentioned in the Introduction we need to work with the homology of the (local) unstable manifold of an isolated invariant set $S$. As mentioned before, in general this set need not be a manifold and in fact it can have a very complicated topology. For instance, the (local) unstable manifold of a horseshoe is a Cantor set times an interval. To overcome this difficulty we will take $H_*$ to be \v{C}ech homology (see \cite[Chapter IX, p. 233ff.]{eilenbergsteenrod1}) instead of singular homology. Very roughly speaking, \v{C}ech homology allows chains to have arbitrarily small discontinuities, and so it may be nonzero even when singular homology is zero. On the category of compact pairs, and with coefficients in a field, \v{C}ech homology satisfies all the usual axioms.

Sometimes it is convenient to use homology with $\mathbb{Z}$ coefficients. The definition of the homology Conley index of an isolated invariant set $S$ works the same. Choosing $(N,L)$ to be a pair of manifolds or polyhedra still ensures that $H_*(N/L,[L])$ is finitely generated but, as mentioned above, for an Abelian group this does not ensure that it has a canonical decomposition and Proposition \ref{prop:W2invariant} may fail. For continuous-time dynamics (semiflows) the induced map $\phi_{N/L}$ is the identity and so the polynomial $x-1$ annihilates every element in $H_*(N/L,[L])$. Thus $H_*(N/L,[L])$ indeed has the (trivial) canonical decomposition $H_*(N/L,[L]) = \{0\} \oplus CH_*(N/L)$. For discrete dynamics it is less common to work with coefficients in $\mathbb{Z}$ but we may do so whenever we consider an invariant set $S$ which has a filtration pair $(N,L)$ such that $H_*(N/L,[L])$ has a canonical decomposition. In either case Proposition \ref{prop:W2invariant} still holds, and this is all we will need in the sequel. We mention here that \v{C}ech homology with coefficients in $\mathbb{Z}$ may fail to satisfy the exactness axiom even for compact pairs. This will not affect our arguments but can be fixed by using \v{C}ech cohomology instead (at the cost of losing some intuition).

\section{A dynamical lemma}

For later use we prove here a dynamical lemma which also provides an illustration of why the notion of an algebraic element is useful.

Suppose $T$ is a compact (metric) space and $g : T \longrightarrow T$ is a continuous mapping. For example, $T$ might be a trapping region in phase space and $g$ the restriction of the dynamics to $T$. Consider the nested sequence of compact spaces $g^k(T)$ and their intersection $K := \bigcap_{k \geq 0} g^k(T)$. Clearly $K$ is a compact invariant set; it is in fact the attractor determined by $T$. We are going to prove the following lemma:

\begin{lemma} \label{lem:attrac} Denote by $j : K \subseteq T$ the inclusion. Then the homomorphism induced by $j$ in (\v{C}ech) homology restricts to an isomorphism $j_*| : AH_*(K) \longrightarrow A H_*(T)$.
\end{lemma}

There exist results in the literature relating the cohomology of an attractor for a homeomorphism and its basin of attraction (\cite{gobbino1}, \cite{pacoyo1}). The main difference here is that we cannot assume the dynamics to be invertible, since we will eventually apply this lemma to the dynamics in a quotient $N/L$.

To prove the lemma we will use the following description of the homology of $K$, which follows from the continuity property of \v{C}ech homology (see for example \cite[Theorem 3.1, p. 261]{eilenbergsteenrod1}). By construction the sets $g^k(T)$ form a decreasing sequence of compact whose intersection is $K$. Then the homology of $K$ can be identified with the inverse limit of the inverse sequence \[\xymatrix{H_*(T) & H_*(g(T)) \ar[l] & H_*(g^2(T)) \ar[l] & \ldots \ar[l]}\] where each arrow is induced by an inclusion. This inverse limit consists by definition of all sequences $(z_k)$ where (i) each $z_k \in H_*(g^k(T))$ and (ii) the $z_k$ are coherent in the sense that each $z_{k+1}$ is sent to $z_k$ by the inclusion induced map $H_*(g^{k+1}(T)) \longrightarrow H_*(g^k(T))$. With this identification, the inclusion induced map $H_*(K) \longrightarrow H_*(g^k(T))$ is just given by $(z_k) \longmapsto z_k$.

\begin{proof}[Proof of Lemma \ref{lem:attrac}] Consider each $g^k(T)$ endowed with the dynamics given by $g$; i.e. with the map $g|_{g^k(T)} : g^k(T) \longrightarrow g^k(T)$. Let $j_k$ be the inclusion of $g^{k+1}(T)$ in $g^k(T)$ and set $r_k : g^k(T) \longrightarrow g^{k+1}(T)$ to be the map $r_k(x) := g(x)$. Clearly $j_k$ and $r_k$ are shift inverses to each other, and so are the homomorphisms they induce in homology. Thus by Proposition \ref{prop:inv_shift}.(i) each homomorphism $(j_k)_* : AH_*(g^{k+1}(T)) \longrightarrow AH_*(g^k(T))$ is an isomorphism.

We now show that $j_*|$ is injective. Let $z = (z_k) \in H_*(K)$ be algebraic and assume that $j_*(z) = 0$; i.e. $z_0 = 0$. Since $z$ is algebraic, so is its image $z_k$ under each inclusion induced map $H_*(K) \longrightarrow H_*(g^k(T))$ because these are all equivariant. Thus $z_1$ is an element in $AH_*(g(T))$ whose image under the inclusion induced map $AH_*(g(T)) \longrightarrow AH_*(T)$ is $z_0 = 0$, so $z_1 = 0$ by the paragraph above. Starting now with $z_1$ and repeating the same argument shows that $z_2 = 0$, and inductively all $z_k = 0$ so $z = 0$.

To prove that $j_*|$ is surjective pick $z \in H_*(T)$ algebraic; say that it is annihilated by the normalized polynomial $P$. Set $z_0 := z$. By the first paragraph of the proof there is an algebraic $z_1 \in H_*(g(T))$ such that $(j_1)_*(z_1) = z_0$. We observe that in fact $z_1$ is annihilated by $P$. Indeed, $P \cdot z_1 \in H_*(g(T))$ is still algebraic (Proposition \ref{prop:alg}.(ii)) and is carried by $(j_1)_*$ to $P \cdot z_0 = 0$, so by the injectivity of $(j_1)_*$ on algebraic elements we must have $P \cdot z_1 = 0$. Starting now with $z_1$ we reason in the same way to construct $z_2 \in AH_*(g^2(T))$ which goes to $z_1$ under the inclusion $(j_2)_*$ and is annihilated by $P$. Inductively, this defines a sequence $z = (z_k) \in H_*(K)$ which goes to $z_0 = z$ via $j_*$. Notice that $P \cdot z = (P \cdot z_k) = 0$, so $z$ is an algebraic element in $H_*(K)$. Thus $j_*|$ is indeed surjective.
\end{proof}

\section{Isolated invariant sets as ``emitters''} \label{sec:emitters}

Let $S$ be an isolated invariant set and suppose $(N,L)$ is a filtration pair for $S$. The local unstable manifold of $N$ is the set $W^u_{\rm loc}(N)$ (or just $W^u_{\rm loc}$) of points $x \in N$ having a negative semiorbit $(x_n)_{n \leq 0}$ entirely contained in $N$. The local unstable manifold is closed in $N$ and hence compact. We denote by $F := W^u_{\rm loc}(N) \cap L$ and call $F$ a domain of the unstable manifold of $S$, although strictly speaking it need not be entirely contained in $W^u(S)$ (see the end of this section). It is ``a domain'' in the sense that its iterates cover the whole unstable manifold of $S$, and it is not ``a fundamental domain'' since it may have large overlaps with its iterates.

Every element $z$ in the Conley index $CH_*(N/L)$ can be represented, essentially by definition, by a relative chain in $N$ whose boundary lies in $L$. The basic result in this section is Proposition \ref{prop:algebraic_rep} which shows that in fact $z$ can be represented by a relative chain in $W^u_{\rm loc}$; since its boundary lies in $L$, it actually lies in $W^u_{\rm loc} \cap L = F$. In this manner we will define a linear map \[\partial_F : CH_q(N/L) \longrightarrow H_{q-1}(F)\] where, roughly, $\partial_F (z)$ is the boundary of a representative of $z$ contained in $W^u_{\rm loc}$. As a physical analogy, every element $z$ in $CH_q(N/L)$ is represented by a ``wave'' (a relative cycle) in the local unstable manifold and its ``wavefront'' (the boundary of the cycle) is precisely $\partial_F (z)$.

\begin{example} Consider the map $f(x,y,z) = (x/2,y/2,2z)$ in $\mathbb{R}^3$. This generates a dynamical system with the phase portrait of Figure \ref{fig:receiver1}.(a). The origin $S$ is a hyperbolic fixed point with the plane $z = 0$ as its stable manifold $W^s$ and the $z$--axis as its unstable manifold $W^u$. Choose the filtration pair $(N,L)$ shown in Figure \ref{fig:receiver1}.(b): $N$ is a cube and $L$ consists of two thick slabs at the top and bottom of the cube. The local unstable manifold is just the portion of the $z$--axis contained in the cube and $F$ consists of two short segments at the ends of $W^u_{\rm loc}$.

Clearly $f_{N/L}$ is homotopic to the identity and so $CH_*(N/L)$ is the homology $H_*(N/L,[L])$ while the index map $\phi_{N/L}$ is the identity. This homology is nontrivial only in degree one, and in this degree it is generated by the homology class $z$ of the relative cycle $g$ shown as a jagged line in the figure. As claimed above $z$ is also represented by a relative chain that is entirely contained in $W^u_{\rm loc}$; namely the whole straight line segment that constitutes $W^u_{\rm loc}$. In this case, then, $\partial_F(z)$ is the $0$-chain in $F$ that consists of the two endpoints of the local unstable manifold. 
\end{example}

\begin{figure}[h!]
\null\hfill
\subfigure[]{
\begin{pspicture}(0,-1)(5,6)
	\rput[bl](0,0){\scalebox{0.8}{\includegraphics{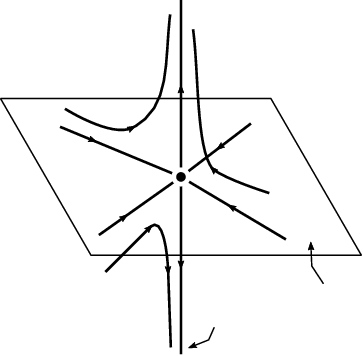}}}
	\rput[bl](2.9,0.4){$W^u$}
	\rput[tl](4.4,0.9){$W^s$}
        \rput[bl](2.1,2.7){$S$}
\end{pspicture}}
\hfill
\subfigure[]{
\begin{pspicture}(0,0)(5,7)
	\rput[bl](0,0){\scalebox{0.8}{\includegraphics{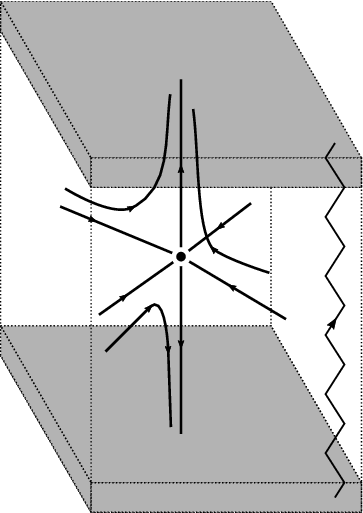}}}
	\rput(4.4,2){$g$}
         \rput[bl](2.1,3.7){$S$}
\end{pspicture}}
\hfill\null
\caption{The set $S$ and a filtration pair $(N,L)$ \label{fig:receiver1}}
\end{figure}
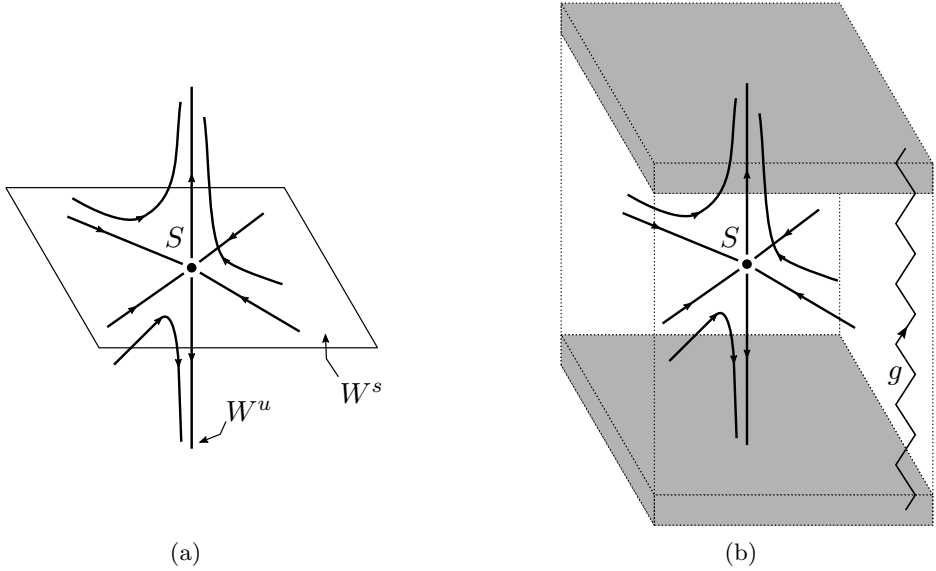

Now we formalize this idea. Instead of working with (\v{C}ech) cycles in $W^u_{\rm loc}$ relative to $F$ it will be easier to work in the quotient $W^u_{\rm loc}/F$. This is the result of collapsing $W^u_{\rm loc} \cap L$ to a single point, so it is the same as the image of $W^u_{\rm loc}$ under the canonical projection $N \longrightarrow N/L$; in particular, $W^u_{\rm loc}/F$ is a compact subset of $N/L$ which is invariant under $f_{N/L}$. It might happen that $F = \emptyset$; then we declare that $W^u_{\rm loc}/F$ is the image of $W^u_{\rm loc}$ under the canonical projection together with the isolated point $[L]$ (this is the same convention used for the quotient $N/L$ when $L = \emptyset$).

\begin{proposition} \label{prop:algebraic_rep} Let $j$ be the inclusion of $W^u_{\rm loc}/F$ into $N/L$. Then $j_*$ restricts to an isomorphism between the algebraic part of $H_q(W^u_{\rm loc}/F,[F])$ and the Conley index $CH_q(N/L)$.
\end{proposition}
\begin{proof} The dynamical key of the argument is that $W^u_{\rm loc}/F$ attracts all of $N/L$ in the sense that the decreasing sequence of compacta \[N/L \supseteq f_{N/L}(N/L) \supseteq \ldots\] has precisely $W^u_{\rm loc}/F$ as its intersection. We first prove this. The inclusion $W^u_{\rm loc}/F \subseteq \bigcap_{k \geq 0} f_{N/L}^k(N/L)$ holds because $W^u_{\rm loc} \subseteq f(W^u_{\rm loc})$. For the converse, suppose $[x] \in f_{N/L}^k(N/L)$ for every $k \geq 0$. If $[x] = [L]$ there is nothing to prove since this point belongs to $W^u_{\rm loc}/F$ by definition. Assume then that $[x] \neq [L]$. Then for each $k$ there exists $x_k \in N \setminus L$ such that $x = f^k(x_k)$, so $x$ has the backwards orbit segment $x_k,f(x_k),\ldots,f^k(x_k)$ entirely contained in $N$. Since $k$ is arbitrarily large, a diagonal argument shows that $x$ has a full negative semiorbit contained in $N$ (see the proof of \cite[Proposition 2.2, p. 3306]{franksricheson1}). Thus indeed $x \in W^u_{\rm loc}$ and so $[x] \in W^u_{\rm loc}/F$ as was to be shown. The proposition then follows from Lemma \ref{lem:attrac} applied to $T = N/L$ and $K = W^u_{\rm loc}/F$ and the equality (by definition) $A H_*(N/L,[L]) = CH_*(N/L)$.
\end{proof}

The proposition above shows that there is a well defined homomorphism ${j_*|}^{-1} : CH_q(N/L) \longrightarrow AH_q(W^u_{\rm loc}/F,[F]) \subseteq H_q(W^u_{\rm loc}/F,[F])$. Also, the canonical projection $\pi : (W^u_{\rm loc},F) \longrightarrow (W^u_{\rm loc}/F,[F])$ induces an isomorphism $\pi_* : H_*(W^u_{\rm loc},F) \cong H_*(W^u_{\rm loc}/F,[F])$ owing to the strong excision property of \v{C}ech homology on compact pairs (\cite[Theorem 5.4, p. 266]{eilenbergsteenrod1}). Notice that this is true even when $F = \emptyset$ because of the \emph{ad hoc} definition for that case. Then we define the homomorphism $\partial_F : CH_q(N/L) \longrightarrow H_{q-1}(F)$ essentially as explained before; namely as the composition \[\partial_F : \xymatrix{CH_q(N/L) \ar[r]^-{{j_*|}^{-1}} & H_q(W^u_{\rm loc}/F,[F]) \cong H_q(W^u_{\rm loc},F) \ar[r]^-{\partial} & H_{q-1}(F)},\] where $\partial$ is the boundary map of the homology long sequence for the pair $(W^u_{\rm loc},F)$.

Later on we will need the following lemma:

\begin{lemma} \label{lem:include_bdry} Let $(N,L)$ and $(N',L')$ be two filtration pairs for the same isolated invariant set $S$. Assume that $(N,L) \subseteq (N',L')$. Denote by $i : F \longrightarrow F'$ the inclusion and by $r : N/L \longrightarrow N'/L'$ the natural shift equivalence given by inclusion followed by projection (see Example \ref{ex:inclusion}). Then the following diagram commutes: \[\xymatrix{CH_k(N/L) 
 \ar[d]_{r_*} \ar[r]^-{\partial_F} & H_{k-1}(F) \ar[d]_{i_*} \\ CH_k(N'/L') \ar[r]^-{\partial_{F'}} & H_{k-1}(F')}\]
\end{lemma}

This lemma is the first of several results in the paper which claim that certain diagrams involving $r_*$ are commutative. They can all be interpreted by recalling from Proposition \ref{prop:W2invariant} that $r_*$ is an isomorphism which can be thought of as a ``change of coordinates'' for the Conley index. With this idea in mind, the lemma above just states that the operator $\partial$ is independent of the ``coordinate system'' $(N,L)$ in which it is expressed.

\begin{proof} Denote by $W^u_{\rm loc}(N)$ and $W^u_{\rm loc}(N')$ the local unstable manifolds of both filtration pairs; notice that $W^u_{\rm loc}(N) \subseteq W^u_{\rm loc}(N')$. Similarly, let $F = W^u_{\rm loc}(N) \cap L$ and $F' = W^u_{\rm loc}(N') \cap L'$ be the corresponding domains, which evidently satisfy $F \subseteq F'$. Consider the following diagram: \[\xymatrix{ CH_k(N/L) \ar[d]_{r_*} & H_k(W^u_{\rm loc}(N)/F,[F]) \ar[l]_-{j_*} \ar[d]_{r_*} & H_k(W^u_{\rm loc}(N),F) \ar[l]_-{\pi_*}^-{\cong} \ar[r]^-{\partial} \ar[d] & H_{k-1}(F) \ar[d]^{i_*} \\ CH_k(N'/L') & H_k(W^u_{\rm loc}(N')/F',[F']) \ar[l]_-{j'_*} & H_k(W^u_{\rm loc}(N'),F') \ar[l]_-{\pi'_*}^-{\cong} \ar[r]^-{\partial'} & H_{k-1}(F')} \] Here $\pi$ and $\pi'$ are the canonical projections and the unlabeled vertical arrow is induced by the obvious inclusion $(W^u_{\rm  loc}(N),F) \subseteq (W^u_{\rm loc}(N'),F')$. The middle square commutes since it already commutes at the level of topological spaces due to the definition of $r$. The rightmost square commutes by the naturality of the long sequence of a compact pair in \v{C}ech homology.

Pick $z \in CH_k(N/L)$. There exists a unique algebraic element $w \in H_k(W^u_{\rm loc}(N)/F,[F])$ such that $j_*(w) = z$ and we have $\partial_F (z) = \partial \circ \pi_*^{-1}(w)$ by definition. Similarly, and considering $r_*(z) \in CH_*(N'/L')$, there exists a unique algebraic $w' \in H_*(W^u_{\rm loc}(N')/F',[F'])$ such that $j'_*(w') = r_*(z)$; again $\partial_{F'}(r_*(z)) = \partial' \circ (\pi')^{-1}_*(w')$ by definition.

Now, the leftmost square commutes, so $r_*(z) = j'_* \circ r_*(w)$. Since $w$ is algebraic and $r$ is equivariant, $r_*(w)$ is also algebraic and therefore by uniqueness $w' = r_*(w)$. The commutativity of the other two squares then implies that $i_*(\partial_F (z)) = \partial_{F'}(r_*(z))$.
\end{proof}

We mentioned in the Introduction that $\partial$ captures the intuitive idea of $S$ emitting wavefronts along its unstable manifold $W^u(S)$. We conclude with a remark about this. It is not always the case that $F \subseteq W^u(S)$, since $L$ may contain an invariant set (this will actually be the case when we study attractor-repeller decompositions in Section \ref{sec:ARdec}). However, one can safely visualize the image of $\partial_F$ as lying in the unstable manifold of $S$. Notice first that the portion $W := W^u_{\rm loc} \cap \overline{N \setminus L}$ of the local unstable manifold is indeed contained in $W^u(S)$. Denote by ${\rm fr}\ L$ the topological frontier of $L$ in $N$ and by ${\rm fr}\ W := W \cap {\rm fr}\ L$, a sort of ``boundary'' of $W$. One then has $W^u_{\rm loc}/F = W/{\rm fr}\ W$ and so in the definition of $\partial_F$ we see that its image is actually contained in $H_{q-1}({\rm fr}\ W)$, where now ${\rm fr}\ W \subseteq W^u(S)$. This accords better with the idea of a wavefront emitted by $S$ and supported in ${\rm fr}\ W$.

\section{Isolated invariant sets as ``receivers''} \label{sec:receivers}

We have seen in the previous section how an invariant set $S$ can ``emit'' elements of its Conley index. Now we explain how $S$ can ``receive'' certain relative cycles to produce elements of its Conley index. This is more involved and will be explained in several stages. In Subsection \ref{subsec:motiva} we make use of the example from the previous section to illustrate the essential ideas. We show how to define some sort of algebraic $\omega$-limit of a chain that lies nearby $S$ and then translate this into a formal definition of a map $\omega_{N/L}$ in Subsection \ref{subsec:omegaNL}. In Subsection \ref{subsec:OmegaNL} we explain how to extend this definition to chains that lie far away from the invariant set, leading to a map $\Omega_{N/L}$ which extends $\omega_{N/L}$.

\subsection{A heuristic description} \label{subsec:motiva} 

Let $S$ be an isolated invariant set with a filtration pair $(N,L)$. The local stable manifold of $N$ is defined as the set $W^s_{\rm loc}(N)$ of points in $N$ whose positive semiorbit is entirely contained in $N \setminus L$; that is, $W^s_{\rm loc}(N) = \{x \in N : f^n(x) \in N \backslash L \text{ for all } n \geq 0\}$. The notation $W^s_{\rm loc}(N)$ is incomplete since it does not record the set $L$ although it plays a role in the definition; we will often abbreviate it by $W^s_{\rm loc}$ anyway. The local stable manifold is compact, positively invariant, and disjoint from $L$. In particular it is homeomorphic to its projection onto $N/L$ and we will use the same notation to denote both sets.

\begin{example} \label{ex:omega} We use again the same dynamics from Figure \ref{fig:receiver1}. The local stable manifold is the portion of the $z = 0$ plane contained in $N$. Recall that $CH_1(N/L)$ is generated by $[g]$.

(A) Consider the two $1$--chains $c_1$ and $c_2$ depicted as jagged lines in panels (a) and (b) of Figure \ref{fig:receiver2}. They can be arbitrary as long as their boundaries $\partial c_1$ and $\partial c_2$ do not intersect the local stable manifold $W^s_{\rm loc}$. The invariant set $S$ and the local stable manifold $W^s_{\rm loc}$ are shown for reference purposes. Regard $c_1$ and $c_2$ as chains in $N/L$ (although we still draw them in $N$ for simplicity) and imagine iterating them forwards with the dynamics $f_{N/L}$ in $N/L$. The two panels of the figure depict four ``snapshots'' of this forward evolution of the chains, showing how they get stretched onto $W^u_{\rm loc}$. We want to provide reasonable definitions for their ``$\omega$-limits'' as elements of $CH_1(N/L)$. Consider first the chain $c_1$. Although this might not be clear in the drawing, $c_1$ is taken to lie entirely below the local stable manifold, and so its iterates converge towards $[L]$. In fact, since (the support of) $c_1$ is a compact set disjoint from the local stable manifold, at some finite time $c_1$ will be entirely absorbed by $[L]$. Therefore it seems reasonable to declare that its limit is the zero element of the Conley index $CH_1(N/L)$. The chain $c_2$ behaves differently, however. This time the chain (and therefore also its iterates) intersects the local stable manifold (the hollow points in the drawing represent these intersections). As the chain evolves with the dynamics it gets stretched in the unstable direction, and its endpoints eventually fall into $[L]$. At that stage the chain will represent an element in $H_1(N/L,[L])$; in fact, it will represent $\pm [g]$ (depending on the chosen orientations). Observe that further iteration of $c_2$ will only ``stretch'' the chain even more, but its endpoints will remain in $[L]$ and it will still represent the same element. Thus it makes sense to say that the limit of $c_2$ is the element $[g] \in CH_1(N/L)$.

\begin{figure}[h!]
\null\hfill
\subfigure[Chain $c_1$]{
\begin{pspicture}(0,0)(5,7)
	\rput[bl](0,0){\scalebox{0.8}{\includegraphics{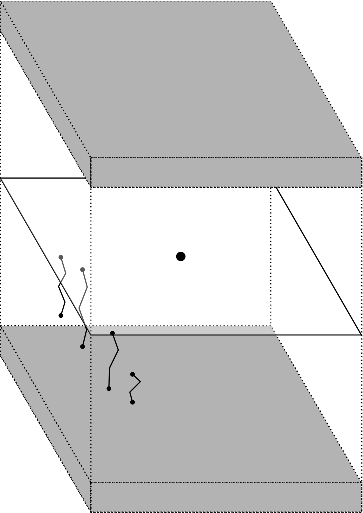}}}
	\rput(0.5,3){$c_1$}
        \rput(4.5,2.1){$W^s_{\rm loc}$}
        \rput[bl](2.7,3.5){$S$}
\end{pspicture}}
\hfill
\subfigure[Chain $c_2$]{
\begin{pspicture}(0,0)(5,7)
	\rput[bl](0,0){\scalebox{0.8}{\includegraphics{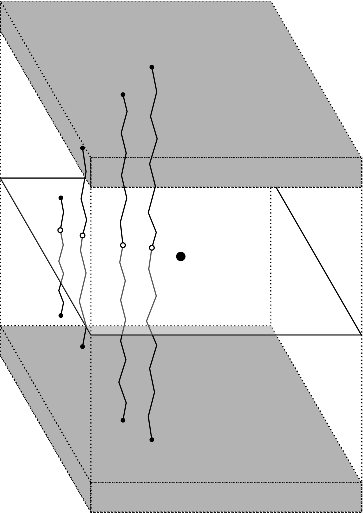}}}
	\rput(0.5,3){$c_2$}
  \rput(4.5,2.1){$W^s_{\rm loc}$}
        \rput[bl](2.7,3.5){$S$}
\end{pspicture}}
\hfill\null
\caption{ \label{fig:receiver2}}
\end{figure}

(B) We now make an observation. Consider again the chain $c_2$. Imagine ``trimming it down'' and leaving just a short portion of the chain around the hollow point where it intersects the local stable manifold. Let us denote by $c'_2$ this new chain. Then when the dynamics is applied to $c'_2$ it should be clear that one may need to iterate the chain for a longer time before its endpoints lie in $[L]$, but thereafter its behaviour is essentially the same as that of the iterates of $c_2$. In particular, their limits should be the same. In other words, one expects that the limit of a chain depends only on the portion of the chain in an arbitrarily small neighbourhood (an ``infinitesimal neighbourhood'') of its intersection with the local stable manifold. This suggests that it should be possible to formalize this construction as a map $\omega_{N/L} : H_*(N/L,N/L \backslash W^s_{\rm loc}) \longrightarrow CH_*(N/L,[L])$.

(C) Let us consider the situation for discrete dynamics. This introduces a new element in the picture. Imagine that the dynamics is given by $f(x,y,z) = (x/2,y/2,-2z)$. The phase portrait of this is essentially the same as before, and in particular $(N,L)$ is still a filtration pair for the origin $S$. 
The Conley index $CH_1(N/L)$ is still generated by $[g]$ but now the index map $\phi_{N/L}$ is $-{\rm Id}$. If we try to replicate the limiting process described above for the chain $c_2$, it is clear that $c_2$ gets stretched as before, but after its endpoints land in $L$ the element $H_1(N/L,[L])$ represented by the chain oscillates between $[g]$ and $-[g]$ and so does not ``converge''.

The way out of this is to analyze slightly more carefully the idea of convergence of a chain $c$ that we wish to capture. In a very crude sense the dynamics on $N/L$ has two ``components''. There is a ``horizontal'' component of the motion in the stable direction (parallel to the $xy$ plane) which moves the chain towards the local unstable manifold (the $z$--axis, in this case). This is the motion whose limit we want to find. There is also a ``vertical'' component of the motion, which is the one that flips $g$ each time the dynamics is applied. We would like to disregard this component of the dynamics, and in order to do so we notice that the local unstable manifold itself is also flipped by the dynamics each time $f$ is applied. This implies that an observer ``anchored'' to the local unstable manifold would still notice that the chain $c_2$ approaches her under the dynamics but would be unaware of the vertical flipping motion. For such an observer, then, the sequence of iterates of the chain would indeed be convergent. This suggests that the limit of a chain should be defined from the perspective of an observer comoving with the local unstable manifold.
\end{example}

We summarize the discussion of the example as follows. Let $c$ be a singular chain in $N/L$ whose boundary $\partial c$ does not meet $W^s_{\rm loc}$. For big enough $k$ we have that the boundary of the chain $f^k_{N/L} \circ c$ collapses onto $[L]$, and so it represents a homology class in $H_*(N/L,[L])$. Recalling that the latter has the canonical decomposition ${\rm gker}\ (f_{N/L})_* \oplus CH_*(N/L)$, by making $k$ even larger we may achieve that $f_{N/L}^k \circ c$ represents an element in $CH_*(N/L)$. We finally define the ``algebraic $\omega$-limit'' of the chain $c$ to be $\phi_{N/L}^{-k}([f_{N/L}^k \circ c])$. The effect of $\phi_{N/L}^{-k}$ is precisely to eliminate the vertical component of the motion as described in (C). Notice that the structure of this expression is essentially the same as that of the canonical projection $\omega$ in p. \pageref{pg:can_proj}. Also, when the dynamics is given by a semiflow we have $\phi_{N/L} = {\rm Id}$ and so this last step is irrelevant.

\subsection{The map $\omega_{N/L}$} \label{subsec:omegaNL} Fix any filtration pair $(N,L)$. We will now define a homomorphism \[\omega_{N/L} : H_*(N/L,N/L \backslash W^s_{\rm loc}) \longrightarrow C H_*(N/L)\] which formalizes the preceding discussion (and generalizes it to \v{C}ech homology).

Regard the local stable manifold $W^s_{\rm loc}$ as a subset of $N/L$. It is straightforward to check that $N/L \setminus W^s_{\rm loc}$ is positively invariant under $f_{N/L}$, so $f_{N/L}$ induces a semidynamical system on the pair $(N/L,N/L \setminus W^s_{\rm loc})$. Let $i : (N/L,[L]) \longrightarrow (N/L,N/L \backslash W^s_{\rm loc})$ be the inclusion. This induces a homomorphism $i_* : H_*(N/L,[L]) \longrightarrow H_*(N/L,N/L \setminus W^s_{\rm loc})$. Both homology groups are equipped with the endomorphism $(f_{N/L})_*$ and $i_*$ is equivariant with respect to these. Recall that $H_*(N/L,[L])$ has the canonical decomposition ${\rm gker}\ (f_{N/L})_* \oplus CH_*(N/L)$. 

\begin{proposition} \label{prop:can_desc_omega} The map $i_*$ restricts to an isomorphism between $CH_*(N/L)$ and its image in $H_*(N/L,N/L \setminus W^s_{\rm loc})$. Moreover, the latter admits the canonical decomposition \[H_*(N/L,N/L \setminus W^s_{\rm loc}) = {\rm gker}\ (f_{N/L})_* \oplus i_* CH_*(N/L).\]
\end{proposition}

Before proving the proposition we make the following observation. By definition points in $N/L \setminus W^s_{\rm loc}$ are those whose forward orbit eventually reaches $[L]$, and so $N/L \setminus W^s_{\rm loc} = \bigcup_{k \geq 0} f_{N/L}^{-k}([L])$. The definition of a filtration pair ensures that the map $f_{N/L}$ has the special property that $f_{N/L}^{-1}([L])$ is a neighbourhood of $[L]$, and so $N/L \setminus W^s_{\rm loc}$ is actually equal to $\bigcup_{k \geq 0} {\rm int}\ f_{N/L}^{-k}([L])$, where the interior is taken in $N/L$. Since these form an increasing sequence of open sets, for any compact $D$ subset of $N/L \setminus W^s_{\rm loc}$ there exists $k \geq 0$ such that $f_{N/L}^k(D) = [L]$.

\begin{proof}[Proof of Proposition \ref{prop:can_desc_omega}] We first prove the proposition when $N/L$ is an ANR. Notice that then its open subset $N/L \setminus W^s_{\rm loc}$ is also an ANR. Thus we may replace \v{C}ech homology with singular homology throughout, which in particular implies that the compact supports property and the exactness of the long exact sequence of a triple hold. We will use both properties in the proof.

Consider the long exact sequence for the triple $(N/L,N/L \backslash W^s_{\rm loc},[L])$: \[\xymatrix{\ldots \ar[r] & H_q(N/L \setminus W^s_{\rm loc},[L]) \ar[r] &  H_q(N/L,[L]) \ar[r]^-{i_*}  & H_q(N/L,N/L \backslash W^s_{\rm loc}) \ar[r] & \ldots }\] Each term of the sequence admits the endomorphism $(f_{N/L})_*$ and all arrows are equivariant by the naturality of the exact sequence of a triple. Every element in $H_*(N/L \backslash W^s_{\rm loc},[L])$ belongs to the generalized kernel of $(f_{N/L})_*$. This follows because any chain $z$ in $H_*(N/L \setminus W^s_{\rm loc},[L])$ has a compact support $D$ and for large enough $k$ we have $f_{N/L}^k(D) = \{[L]\}$ as discussed before the proof, so that $(f_{N/L})^k_*(z) = 0$. Thus $H_*(N/L \setminus W^s_{\rm loc},[L])$ admits a canonical decomposition with its algebraic part $A H_*(N/L \setminus W^s_{\rm loc},[L]) = 0$. We also have that $H_*(N/L,[L])$ has a canonical decomposition with its algebraic part being $CH_*(N/L)$. Thus by Proposition \ref{prop:exact} the group $H_*(N/L,N/L \setminus W^s_{\rm loc})$ also admits a canonical decomposition and the sequence of algebraic parts \[\xymatrix{\ldots \ar[r] & \underbrace{AH_q(N/L \setminus W^s_{\rm loc},[L])}_{=0} \ar[r] &  \underbrace{AH_q(N/L,[L])}_{=CH_q(N/L)} \ar[r]^-{i_*}  & AH_q(N/L,N/L \backslash W^s_{\rm loc}) \ar[r] & \ldots }\] is exact. Hence $i_*$ restricted to $CH_*(N/L)$ is an isomorphism onto its image and the canonical decomposition of $H_*(N/L,N/L \setminus W^s_{loc})$ is ${\rm gker}\ (f_{N/L})_* \oplus i_* CH_*(N/L)$.

We conclude by extending the result to an arbitrary filtration pair $(N',L')$. Pick a shift equivalence $r : N/L \longrightarrow N'/L'$ where $(N,L)$ is a filtration pair comprised of ANRs. We check that $r$ carries $N/L \setminus W^s_{\rm loc}(N)$ into $N'/L' \setminus W^s_{\rm loc}(N')$, where we have amplified the notation $W^s_{\rm loc}$ to record what filtration pair it refers to. Pick $[x] \in N/L \setminus W^s_{\rm loc}(N)$; this means that its forward orbit under $f_{N/L}$ eventually becomes $[L]$. Then the forward orbit of $r([x])$ under $f_{N'/L'}$ eventually becomes $r([L]) = [L']$, and so $r([x])$ does not belong to $W^s_{\rm loc}(N')$. Thus $r$ provides a shift equivalence $(N/L,N/L \setminus W^s_{\rm loc}(N)) \sim (N'/L',N'/L' \setminus W^s_{\rm loc}(N'))$ which in turn descends to a shift equivalence in \v{C}ech homology. By Proposition \ref{prop:inv_shift} we then have that $H_*(N'/L',N'/L' \setminus W^s_{\rm loc}(N'))$ has a canonical decomposition ${\rm gker}\ (f_{N'/L'})_* \oplus r_* i_* CH_*(N/L)$. Now, there is evidently a commutative square \[\xymatrix{(N/L,[L]) \ar[r]^-i \ar[d]_r & (N/L,N/L \setminus W^s_{\rm loc}(N)) \ar[d]^ r \\ (N'/L',[L']) \ar[r]^-{i'} & (N'/L',N'/L' \setminus W^s_{\rm loc}(N'))} \] which passing to homology shows that $r_* i_* CH_*(N/L) = i'_* r_* CH_*(N/L)$, and since $r_* CH_*(N/L) = CH_*(N'/L')$ by Proposition \ref{prop:W2invariant}, we get the desired result that $H_*(N'/L',N'/L' \setminus W^s_{\rm loc}(N'))$ has the canonical decomposition ${\rm gker}\ (f_{N'/L'})_* \oplus i'_* CH_*(N'/L')$.
\end{proof}

\begin{remark} \v{C}ech homology does not satisfy in general (for noncompact spaces) neither the compact supports property nor the exactness of the sequence of a triple. This is the reason why to prove the proposition we needed to hinge on an ANR pair first.
\end{remark}

Denote by $i_*| : CH_*(N/L) \longrightarrow i_*CH_*(N/L)$ the isomorphism obtained by restricting $i_*$ and by $\omega$ the canonical projection of $H_*(N/L,N/L \backslash W^s_{\rm loc})$ onto $i_* CH_*(N/L)$. Then we define $\omega_{N/L}$ as $(i_*|)^{-1} \circ \omega$; i.e. as the unique homomorphism which makes the following diagram commute: \[\xymatrix{& H_*(N/L,N/L \backslash W^s_{\rm loc}) \ar[d]^{\omega} \ar[dl]_{\omega_{N/L}} \\ CH_*(N/L) \ar[r]^-{i_*|}_-{\cong} & i_* CH_*(N/L) }\]

Using the general formula for the canonical projection in p. \pageref{pg:can_proj} it is easy to check that this definition agrees with the semiformal one given above. This construction also suggests an intuitive interpretation of the elements of the homological Conley index $CH_*(N/L)$ as ``algebraic $\omega$-limits'' of chains in $(N/L,N/L \backslash W^s_{\rm loc})$: the limit of such a chain is an element in $CH_*(N/L)$ by the definition of $\omega_{N/L}$, and every element in $CH_*(N/L)$ is the limit of some chain (namely, of its image under $i_*$).

We conclude by showing that the limit of a chain is ``independent'' of the filtration pair. Suppose $(N,L)$ and $(N',L')$ are two filtration pairs for the isolated invariant set $S$ and denote by $W^s_{\rm loc}(N)$ and $W^s_{\rm loc}(N')$ their local stable manifolds. We compare the two filtration pairs via a shift equivalence $r : N/L \longrightarrow N'/L'$ and prove the following:

\begin{proposition} \label{prop:indep_omegaNL} There is a commutative diagram \[\xymatrix{H_*(N/L,N/L \backslash W^s_{\rm loc}(N)) \ar[r]^-{r_*} \ar[d]_{\omega_{N/L}} & H_*(N'/L',N'/L' \backslash W^s_{\rm loc}(N')) \ar[d]_{\omega_{N'/L'}} \\ CH_*(N/L) \ar[r]^-{r_*} & CH_*(N'/L')}\]
\end{proposition}
\begin{proof} Recall from the proof of Proposition \ref{prop:can_desc_omega} that $r$ provides a shift equivalence between the pairs $(N/L,N/L \setminus W^s_{\rm loc})$ and $(N'/L',N'/L' \setminus W^s_{\rm loc})$, so that the above diagram makes sense.

Denote by $i : (N/L,[L]) \subseteq (N/L,N/L \backslash W^s_{\rm loc}(N))$ the inclusion. The composition $i_* \circ \omega_{N/L}$ is by definition just the canonical projection of $H_*(N/L, N/L \backslash W^s_{\rm loc}(N))$ onto $i_* CH_*(N/L)$, and similarly for the pair $(N',L')$. Now, since $r_*$ is an equivariant homomorphism, it preserves canonical decompositions. Thus there exists a commutative diagram \[ \xymatrix{H_*(N/L,N/L \backslash W^s_{\rm loc}(N)) \ar[r]^-{r_*} \ar[d]_{i_* \circ \omega_{N/L}} & H_*(N'/L',N'/L' \backslash W^s_{\rm loc}(N')) \ar[d]_{i'_* \circ \omega_{N'/L'}} \\ i_* CH_*(N/L) \ar[r]^-{r_*|} & i'_* CH_*(N'/L')) }\] where $r_*|$ simply means the appropriate restriction of the map $r_*$. Tacking on the commutative triangles that define $\omega$ onto the sides we have another commutative diagram \[ \xymatrix{& H_*(N/L,N/L \backslash W^s_{\rm loc}(N)) \ar[r]^-{r_*} \ar[d]_{i_* \circ \omega_{N/L}} \ar[dl]_{\omega_{N/L}} & H_*(N'/L',N'/L' \backslash W^s_{\rm loc}(N')) \ar[d]_{i'_* \circ \omega_{N'/L'}} \ar[dr]^{\omega_{N'/L'}} & \\ CH_*(N/L) \ar[r]^-{i_*|}_-{\cong} &  i_* CH_*(N/L) \ar[r]^-{r_*|} & i'_* CH_*(N'/L') & CH_*(N'/L') \ar[l]_-{i'_*|}^-{\cong} }\] which proves the proposition.
\end{proof}

\subsection{The $\Omega_{N/L}$--limit map} \label{subsec:OmegaNL} So far we have discussed how to define the limit of a chain $c$ which lies in the vicinity of the invariant set $S$. Now we generalize this to chains that lie anywhere in phase space.

\begin{example} Let $S$ and $(N,L)$ be as in Example \ref{ex:omega}. Refer to Figure \ref{fig:receiver3}.(a). We have represented part of the stable manifold of $S$ that extends beyond the isolating neighbourhood $N$; a couple of trajectories contained in it and converging to $S$ are depicted to suggest this. Consider a chain $c$ very far away from $S$ (in particular, outside of $N$). We try to replicate the process described above by taking iterates of the chain. The part of the chain that is contained in the stable manifold of $S$ will eventually enter $N$, but there is no reason to expect that the rest of the chain will enter $N$ as well. This is suggested in Figure \ref{fig:receiver3}.(a): by the $n$th iterate the hollow points (the intersection of the iterates of the chain with the stable manifold) enter $N$, but the iterated chain $f^n \circ c$ itself is so long that it is not contained in $N$. Looking back at (B) in Example \ref{ex:omega} we see how to overcome this difficulty. As explained there, the limit of a chain should only depend on the portion of the chain that lies in an infinitesimal neighbourhood of the local stable manifold. Therefore we could start by trimming the chain $c$ down to a tiny portion $c'$ near the stable manifold; so small that its $n$th iterate is indeed contained in $N$. This is illustrated in Figure \ref{fig:receiver3}.(b), where the original $c$ is shown in dashed and $c'$ in solid linestyles. Thereafter we know how to proceed: we can then regard the chain $f^n \circ c'$ as an object in $N/L$ and compute its limit via the map $\omega_{N/L}$. Since we did an extra $n$ steps to enter $N$, we must undo them after the limit has been taken.
\end{example}



\begin{figure}[h!]
\null\hfill
\subfigure[]{
\begin{pspicture}(0,0)(12.5,9)
	\rput[bl](0,0){\scalebox{0.8}{\includegraphics{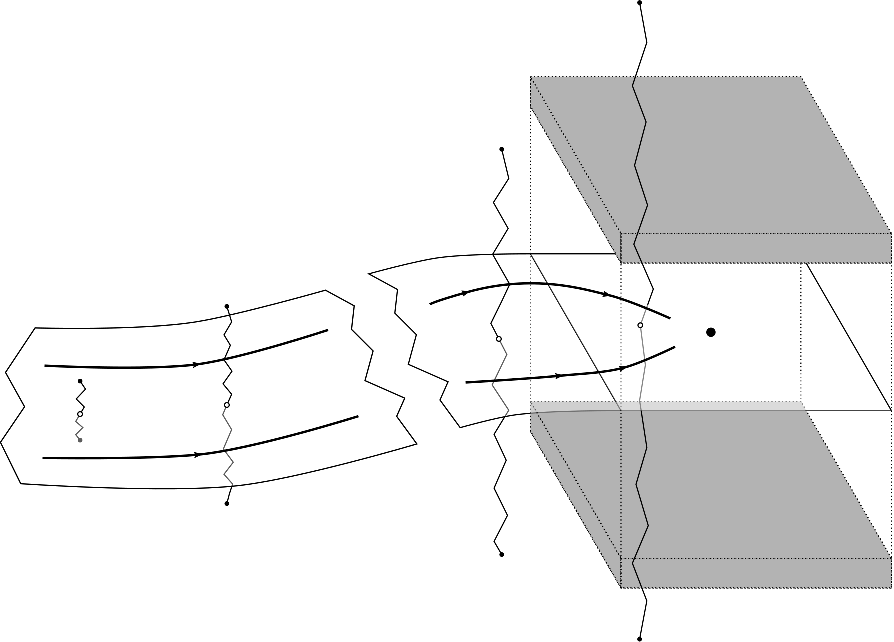}}}
	\rput(0.8,3){$c$}
	\rput(5,2){$W^s(S)$}
        \rput[b](10,4.2){$S$}
	\rput[r](8.5,0.3){$f^k \circ c$}
\end{pspicture}}
\hfill
\subfigure[]{
\begin{pspicture}(0,0)(12.5,9)
	\rput[bl](0,0){\scalebox{0.8}{\includegraphics{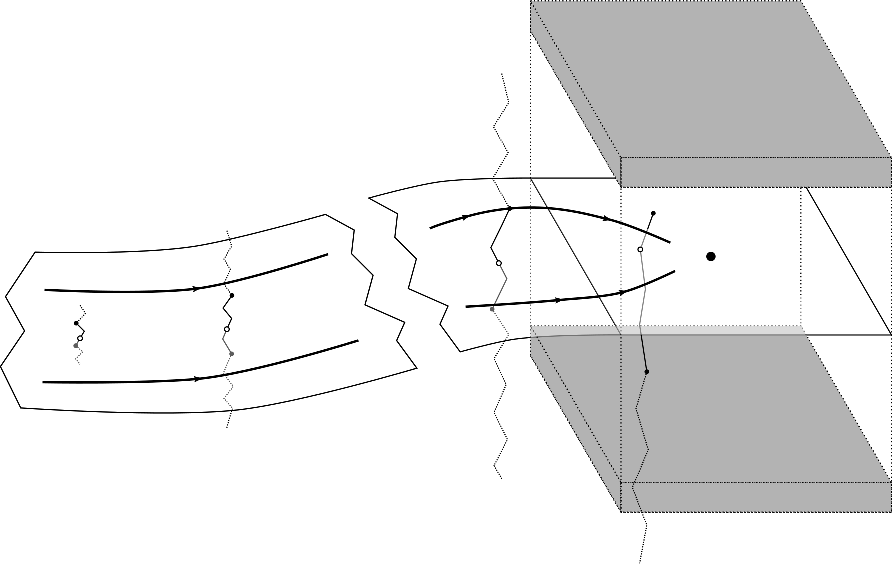}}}
	\rput(0.8,3){$c'$}
	\rput(5,2){$W^s(S)$}
         \rput[b](10,4.2){$S$}
	\rput[l](9,2.6){$f^k \circ c'$}
\end{pspicture}}
\hfill\null
\caption{\label{fig:receiver3}}
\end{figure}

\medskip

{\it Admissible pairs.} The delicate point in formalizing this consists in defining what chains $c$ ``converge'' to $S$. The situation considered in the example above is slightly misleading in this regard because of its simplicity. In essence, for a chain $c$ to have a well defined limit we need that it (rather, its support $Z := |c|$) has a distinguished subset $Z_0$ that converges to $S$ ``head on'' and such that any infinitesimal neighbourhood of that subset gets stretched along the local unstable manifold of $S$. In the example above this distinguished subset $Z_0$ was just the intersection of $|c|$ with the stable manifold of $S$, but this does not always work. Consider the phase portrait of Figure \ref{fig:exnontrivial}. It shows a saddle point $S$ and an isolating neighbourhood $N$ for it. Let $c$ be the $1$--chain shown as a jagged line and notice that, unlike in the simple example given above, $|c| \cap W^s(S)$ does not converge to $S$ in any reasonable sense because it gets stretched onto the orbit homoclinic to $S$. However, there is a subset of $|c|$ which has the required properties; namely the singleton $Z_0$ represented as a hollow point in the drawing. The point itself gets attracted by $S$, and a tiny neighbourhood $U$ of this point in $Z = |c|$ will get stretched onto the local unstable manifold of $S$. If we trim $c$ down to a chain $c'$ by removing the part outside $U$ we can still carry out the same procedure described in Example \ref{ex:omega}. (In this case $CH_1(N/L)$ is generated by a single element and, at least intuitively, the limit of the chain $c$ is precisely this generator).

\begin{figure}[h!]
\begin{pspicture}(0,0)(7.5,5.5)
	\rput[bl](0,0){\scalebox{0.8}{\includegraphics{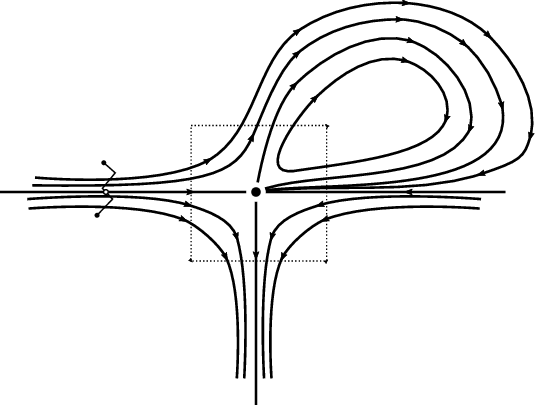}}}
    \rput[tl](4.5,2){$N$}
    \rput[bl](1.4,3.4){$c$}
    \rput[bl](3.1,2.5){$S$}
\end{pspicture}
\caption{\label{fig:exnontrivial}}
\end{figure}

The following definition captures these ideas. For technical reasons it is cast for compact pairs $(Z,Z_0)$ instead of supports of chains. Fix a filtration pair $(N,L)$ for $S$. A pair $(Z,Z_0)$ of compact subsets of phase space is \emph{admissible} (for $S$) if there exist $n \geq 0$ and a neighbourhood $U$ of $Z_0$ in $Z$ such that:
\begin{itemize}
	\item[(A1)] $f^n(Z_0) \subseteq W^s_{\rm loc}$.
	\item[(A2)] $f^n(U \backslash Z_0) \subseteq N \backslash W^s_{\rm loc}$.
\end{itemize}

We will see in Theorem \ref{teo:indepOmegaNL} that the notion of admissibility depends only on $S$ and not the filtration pair $(N,L)$. The following remark shows that the asymptotic behaviour of the definition is as expected:

\begin{remark} \label{rem:posterior} If the admissibility conditions (A1) and (A2) are satisfied for some $n$ and $U$, then for any $n' \geq n$ there exists $U' \subseteq U$ such that $m$ and $V$ satisfy the conditions as well.
\end{remark}
\begin{proof} Observe first that since $f^n(Z_0) \subseteq W^s_{\rm loc}$ and the local stable manifold is positively invariant and disjoint from $L$, we have in fact that $f^n(Z_0), f^{n+1}(Z_0), \ldots, f^{n'}(Z_0)$ are all disjoint from $L$. By the continuity of $f$ there exists a neighbourhood $U'$ of $Z_0$ in $Z$ such that $f^n(U'), f^{n+1}(U'), \ldots, f^{n'}(U')$ are all disjoint from $L$ as well. Clearly we may take $U'$ to be contained in $U$ and then $f^n(U') \subseteq f^n(U) \subseteq N$. Since $f^n(U')$ is disjoint from the exit set $L$ by construction, its image under $f$ still lies in $N$; that is, $f^{n+1}(U') \subseteq N$. Again this set is disjoint from $L$ by assumption, and so $f^{n+2}(U') \subseteq N$. Repeating this argument we see inductively that $f^n(U'),\ldots, f^{n'}(U') \subseteq N$. Since $f^n(U' \setminus Z_0) \subseteq f^n(U \setminus Z_0) \subseteq N \setminus W^s_{\rm loc}$, the positive invariance of $N \setminus W^s_{\rm loc}$ in $N$ then implies that $f^{n'}(U' \setminus Z_0) \subseteq N \setminus W^s_{\rm loc}$ as well.
\end{proof}

The formal definition of $\Omega_{N/L}$ for chains in $(Z,Z \setminus Z_0)$, or rather elements from $H_*(Z,Z \setminus Z_0)$, is now as follows. Clearly in the definition of an admissible pair we can assume $U$ to be compact by reducing it slightly if necessary. Then we can excise the open set $Z \backslash U$ from the pair $(Z,Z \backslash Z_0)$ to obtain an inclusion induced isomorphism $H_*(U,U \backslash Z_0) \cong H_*(Z,Z \backslash Z_0)$. (The excision axiom is valid for \v{C}ech homology without restrictions; see \cite[Theorem 6.1, p. 243]{eilenbergsteenrod1}). Denote by $\pi : (N,N \backslash W^s_{\rm loc}) \longrightarrow (N/L,N/L \backslash W^s_{\rm loc})$ the quotient map and regard $f^n$ as a map $f^n : (U,U \backslash Z_0) \longrightarrow (N,N \backslash W^s_{\rm loc})$. 

\begin{definition} The map $\Omega_{N/L}$ is defined as the composition \[\xymatrixcolsep{3pc} \xymatrix{H_*(Z, Z \backslash Z_0) & H_*(U, U \backslash Z_0) \ar[l]_-{\cong} \ar[r]^-{\pi_* \circ f_*^n} & H_*(N/L, N/L \backslash W^s_{\rm loc}) \ar[rr]^-{\phi_{N/L}^{-n} \circ \omega_{N/L}} & & CH_*(N/L)}\]
\end{definition}

Observe that $(N,W^s_{\rm loc})$ is always admissible (taking $n = 0$ and $U = N$ in the definition) and for that pair the definition of $\Omega_{N/L}$ reduces to $\omega_{N/L}$, so that indeed $\Omega$ extends the local notion of a limit of a chain.

We need to show that the definition is correct; i.e. for a given admissible pair $(Z,Z_0)$ the map $\Omega_{N/L}$ does not depend on the choices of $n$ and $U$:

\begin{proposition} The map $\Omega_{N/L}$ is independent of both $n$ and $U$ (as long as they satisfy the admissibility conditions).
\end{proposition}
\begin{proof} For this proof only we adopt the notation $\Omega_{\{n,U\}}$ to record the choice of $n$ and $U$ in the definition of $\Omega_{N/L}$. We need to show that $\Omega_{\{n,U\}} = \Omega_{\{n',U'\}}$ whenever $\{n,U\}$ and $\{n',U'\}$ satisfy the admissibility conditions.

It is straightforward to check that if $\{n,U\}$ satisfies the admissibility conditions and $V \subseteq U$ then $\{n,V\}$ also satisfies the conditions and $\Omega_{\{n,U\}} = \Omega_{\{n,V\}}$ because the first two arrows in the definition $\Omega_{N/L}$ commute with inclusion induced maps. Thus replacing $U$ and $U'$ with $U \cap U'$ we may assume that $U = U'$ and $n' \geq n$.

By Remark \ref{rem:posterior} there exists a neighbourhood $U'$ of $Z_0$ in $Z$ contained in $U$ such that $\{n',U'\}$ satisfies the admissibility conditions. Since $U' \subseteq U$, again we have $\Omega_{\{n,U\}} = \Omega_{\{n,U'\}}$ and we only need to show that $\Omega_{\{n,U'\}} = \Omega_{\{n',U'\}}$. By construction the first $n'-n$ iterates of any point $x \in f^n(U')$ remain in $N$, and so we have $f_{N/L}^{n'-n}([x]) = [f^{n'-n}(x)]$. Thus the following diagram commutes: \[\xymatrix{(U',U' \backslash Z_0) \ar[r]^-{f^n} \ar[d]^{\rm Id} & (N,N \backslash W^s_{\rm loc}) \ar[r]^-{\pi} & (N/L,N/L \backslash W^s_{\rm loc}) \ar[d]^{f_{N/L}^{n'-n}} \\ (U',U' \backslash Z_0) \ar[r]^-{f^{n'}} & (N,N \backslash W^s_{\rm loc}) \ar[r]^-{\pi} & (N/L,N/L \backslash W^s_{\rm loc})}\] Passing to (\v{C}ech) homology we see that $(f_{N/L})_*^{n'-n} \circ \pi_* \circ f^n_* = \pi_* \circ f^{n'}_*$. Composing both sides on the left with $\omega_{N/L}$ and using the equivariance relation $\omega_{N/L} \circ (f_{N/L})_* = \phi_{N/L} \circ \omega_{N/L}$ yields \[\phi_{N/L}^{n'-n} \circ \omega_{N/L} \circ \pi_* \circ f^n_* = \omega_{N/L} \circ \pi_* \circ f^{n'}_*.\] Composing this again on the left with $\phi_{N/L}^{-n'}$ and recalling the definition of $\Omega$ shows that $\Omega_{\{n,U'\}} = \Omega_{\{n',U'\}}$.
\end{proof}

To close this section we show that the notion of an admissible pair and the limit of a chain are essentially independent of the filtration pair:

\begin{theorem} \label{teo:indepOmegaNL} The admissibility condition on a pair $(Z,Z_0)$ is independent of the filtration pair. If $(Z,Z_0)$ is admissible and $r : N/L \longrightarrow N'/L'$ is a standard shift equivalence with semi-lag $k$ (Example \ref{ex:standard}) then there is a commutative diagram \[\xymatrix{ & H_*(Z,Z \setminus Z_0) \ar[ld]_{\Omega_{N/L}} \ar[rd]^{\Omega_{N'/L'}} & \\ CH_*(N/L) \ar[rr]^{\phi_{N'/L'}^{-k} \circ r_*} & & CH_*(N'/L')}\]
\end{theorem}

Recalling the paragraph after Proposition \ref{prop:W2invariant}, the theorem essentially says that the maps $\Omega_{N/L}$ and $\Omega_{N'/L'}$ are the same operator expressed in two different ``coordinate systems'' (afforded by the filtration pairs) save for a shift in the origin of times (due to the $\phi_{N'/L'}^{-k}$ factor).

\begin{proof}[Proof of Theorem \ref{teo:indepOmegaNL}] First we show that if $(Z,Z_0)$ is admissible for $(N,L)$ then it is also admissible for $(N',L')$. An entirely symmetric argument proves the converse, showing that the admissibility condition is independent of the filtration pair. Denote by $\pi : N \longrightarrow N/L$ the canonical projection and similarly for $\pi'$.

Let $(Z,Z_0)$ be admissible for $(N,L)$. Then there exist $n$ and a neighbourhood $U$ of $Z_0$ in $Z$ such that $f^n(Z_0) \subseteq W^s_{\rm loc}(N)$ and $f^n(U \backslash Z_0) \subseteq N \backslash W^s_{\rm loc}(N)$. Projecting onto $N/L$ and identifying as usual the local stable manifold with its image in $N/L$ we then have that $\pi \circ f^n(Z_0) \subseteq W^s_{\rm loc}(N)$ and $\pi \circ f^n(U \backslash Z_0) \subseteq N/L \backslash W^s_{\rm loc}(N)$. Consider the set $r^{-1}([L']) \subseteq N/L$. This is a compact set disjoint from $W^s_{\rm loc}(N)$ since $r(W^s_{\rm loc}(N)) \subseteq W^s_{\rm loc}(N')$ and the latter set does not contain $[L']$. In particular $\pi \circ f^n(Z_0)$ is disjoint from $r^{-1}([L'])$ and, by choosing $U$ sufficiently small, we can achieve $\pi \circ f^n(U) \cap r^{-1}([L']) = \emptyset$ as well. Since $r$ is a standard shift equivalence, we then have that $r$ acting on any point $[x]$ with $x \in f^n(U)$ is just $[f^k(x)]$. In other words, $r \circ \pi|_{f^n(U)} = \pi' \circ f^k|_{f^n(U)}$. Obviously this relation holds also for the restrictions $|_{f^n(Z_0)}$ and $|_{f^n(U \backslash Z_0)}$ since $Z_0 \subseteq U$.

Starting with $\pi \circ f^n(Z_0) \subseteq W^s_{\rm loc}(N)$, applying $r$ to both sides and using the above relation we get \[r(\pi \circ f^n(Z_0)) \subseteq r(W^s_{\rm loc}(N)) \subseteq W^s_{\rm loc}(N') \ \Rightarrow \ \pi' \circ f^k(f^n(Z_0)) \subseteq W^s_{\rm loc}(N'),\] that is, $f^{n+k}(Z_0) \subseteq W^s_{\rm loc}(N')$. An analogous computation, this time starting with $\pi \circ f^n(U \backslash Z_0) \subseteq N \backslash W^s_{\rm loc}(N')$, yields $\pi' \circ f^{n+k}(U \backslash Z_0) \subseteq N' \backslash W^s_{\rm loc}(N')$. These show that $(Z,Z_0)$ is admissible for $(N',L')$.

The argument of the previous paragraph in fact shows a little more: regarding now the canonical projection $\pi$ as a map of pairs $\pi : (N,N \backslash W^s_{\rm loc}(N)) \longrightarrow (N/L,N/L \backslash W^s_{\rm loc}(N))$ and similarly for $\pi'$, there is a commutative diagram \[\xymatrix{(f^{n}(U), f^{n}(U \backslash Z_0)) \ar[r]^-{f^k} \ar[d]_{\pi} & (f^{n+k}(U),f^{n+k}(U \backslash Z_0)) \ar[d]_{\pi'} \\ (N/L,N/L \backslash W^s_{\rm loc}(N)) \ar[r]^-r & (N'/L',N'/L' \backslash W^s_{\rm loc}(N'))}\] We may enlarge this as follows: \[ \xymatrix{(U,U \backslash Z_0) \ar[d]_{f^{n}} \ar[r]^{\rm Id}& (U,U \backslash Z_0) \ar[d]_{f^{n+k}} \\ (f^{n}(U), f^{n}(U \backslash Z_0)) \ar[r]^-{f^k} \ar[d]_{\pi} & (f^{n+k}(U),f^{n+k}(U \backslash Z_0)) \ar[d]_{\pi'} \\ (N/L,N/L \backslash W^s_{\rm loc}(N)) \ar[r]^r & (N'/L',N'/L' \backslash W^s_{\rm loc}(N'))}\] Passing to homology we see the following. Given $z \in H_*(U,U \setminus Z_0)$, by definition of $\Omega$ we have $\Omega_{N/L}(z) = \phi_{N/L}^{-n} \circ \omega_{N/L}(\pi_* \circ f^n_*(z))$ and similarly $\Omega_{N'/L'}(z) = \phi_{N'/L'}^{-(n+k)} \circ \omega_{N'/L'}(\pi'_* \circ f^{n+k}_*(z))$. Because of the commutative diagram above, the latter is $\phi_{N'/L'}^{-(n+k)} \circ \omega_{N'/L'}(r_* \circ \pi_* \circ f^n_*(z))$. By Proposition \ref{prop:indep_omegaNL} this is $\phi_{N'/L'}^{-(n+k)} \circ r_* \circ \omega_{N/L}(\pi_* \circ f_*^n(z))$, which in turn is $\phi_{N'/L'}^{-(n+k)} \circ r_* \circ \phi_{N/L}^n \circ \Omega_{N/L}(z)$. By the equivariance of $r_*$ this is $\phi_{N'/L'}^{-(n+k)} \circ \phi_{N'/L'}^n \circ  r_* \circ \Omega_{N/L}(z) = \phi_{N'/L'}^{-k} \circ r_* \circ \Omega_{N/L}(z)$. Thus $\Omega_{N'/L'} = \phi_{N'/L'}^{-k} \circ r_* \circ \Omega_{N/L}$, as was to be shown.
\end{proof}

We conclude with a simple remark. \label{pg:remark} We have in fact constructed a \emph{family} of maps $\Omega$, one for each admissible pair $(Z,Z_0)$. It should be intuitively clear that these maps are all compatible with each other in the sense that if a chain is supported by two different pairs, computing its limit with one or the other will result in the same element of the Conley index. We will only make use of the following special case. If $(Z,Z_0)$ is an admissible pair and $Z' \subseteq Z$ is compact, then the pair $(Z',Z' \cap Z_0)$ is also admissible and there is a commutative square \[ \xymatrix{H_*(Z',Z' \setminus Z_0) \ar[r]^{\Omega'} \ar[d]_{j_*} & C H_*(N/L) \ar[d]^{\rm Id} \\ H_*(Z,Z \setminus Z_0) \ar[r]^{\Omega} & C H_*(N/L)} \] where the arrow $j_*$ is induced by the inclusion of the primed pair into the unprimed one and the primed and unprimed omegas denote the corresponding limit maps. This is straightforward to check: if $n$ and $U$ satisfy the admissibility conditions for $(Z,Z_0)$ then $n$ and $U' := U \cap Z'$ satisfy the admissibility conditions for $(Z',Z' \cap Z_0)$ and computing $\Omega = \Omega_{\{n,U\}}$ and $\Omega' := \Omega'_{\{n,U'\}}$ with these yields immediately the above diagram.

\section{The connection homomorphism of an attractor-repeller decomposition} \label{sec:ARdec}

Now we finally apply the ``emitter'' and ``receiver'' ideas developed above to obtain a description of the connection homomorphism of an attractor-repeller decomposition. We begin by recalling some definitions. Suppose $S$ is an isolated invariant set which contains two disjoint (and nonempty) isolated invariant sets $A$ and $R$ such that every orbit $(x_n)_{n \in \mathbb{Z}}$ in $S$ satisfies one of the following mutually exclusive conditions:
\begin{itemize}
	\item[(i)] $(x_n)_{n \in \mathbb{Z}}$ is entirely contained in $A$, or
	\item[(ii)] $(x_n)_{n \in \mathbb{Z}}$ is entirely contained in $R$, or
	\item[(iii)] $(x_n)_{n \in \mathbb{Z}}$ is disjoint from $A$ and $R$ and converges to $A$ as $n \rightarrow +\infty$ and to $R$ as $n \rightarrow -\infty$.
\end{itemize}

In case (iii) ``converges to $A$'' means that the orbit eventually enters (and remains in) any prescribed neighbourhood of $A$; similarly for $R$. The pair $\{A,R\}$ is called an attractor-repeller decomposition of $S$, and the orbits in case (iii) are called the connecting orbits (from $R$ to $A$).

As mentioned in the Introduction, for any attractor-repeller decomposition there is an exact sequence that relates the Conley indices of $A$, $R$ and $S$. There exists a triple of compact spaces $N_0 \supseteq N_1 \supseteq N_2$ such that $(N_0,N_2)$ is a filtration pair for $S$, $(N_0,N_1)$ is a filtration pair for $R$, and $(N_1,N_2)$ is a filtration pair for $A$ (see \cite[Section 7, pp. 3319ff.]{franksricheson1}). Consider the triple $N_0/N_2 \supseteq N_1/N_2 \supseteq N_2/N_2$. Its long exact sequence reads \[\xymatrix{\ldots \ar[r] & H_*(N_1/N_2,[N_2]) \ar[r] & H_*(N_0/N_2,[N_2]) \ar[r] & H_*(N_0/N_2,N_1/N_2) \ar[r] & \ldots}\] where the unlabeled arrows are inclusion induced. By the strong excision property of \v{C}ech homology there is an isomorphism $H_*(N_0/N_2,N_1/N_2) = H_*(N_0/N_1,[N_1])$ induced by the canonical projection $(N_0/N_2)/(N_1/N_2) \longrightarrow N_0/N_1$ and so we may rewrite the above as \[\xymatrix{\ldots \ar[r] & H_*(N_1/N_2,[N_2]) \ar[r] & H_*(N_0/N_2,[N_2]) \ar[r] & H_*(N_0/N_1,[N_1]) \ar[r]^-{\Gamma} & \ldots}\] The space $N_0/N_2$ is endowed with the dynamics given by $f_{N_0/N_2}$. Its subset $N_1/N_2$ is positively invariant and $f_{N_0/N_2}$ coincides there with $f_{N_1/N_2}$. Finally, under the identification $(N_0/N_2)/(N_1/N_2) = N_0/N_1$ the map $f_{N_0/N_2}$ becomes $f_{N_0/N_1}$. Thus in the preceding sequence each group is endowed with its corresponding $(f_{N_i/N_j})_*$ and all the arrows are equivariant. Passing to the algebraic parts, by Proposition \ref{prop:exact} we have an exact sequence \[ \xymatrix{\ldots \ar[r] & CH_*(N_1/N_2) \ar[r] & CH_*(N_0/N_2) \ar[r] & CH_*(N_0/N_1) \ar[r]^-{\Gamma} & \ldots}\] with the same arrows. The map $\Gamma$ is the connection homomorphism of the attractor-repeller decomposition.\footnote{Strictly speaking $\Gamma$ depends on the triple $(N_0,N_1,N_2)$, but in fact it is an ``expression in coordinates'' of a morphism that depends only on the attractor-repeller decomposition. Formalizing this is not completely straightforward; for the case of flows see for example \cite[Theorem 5.7, p. 23]{salamon1} and the references therein.} This sequence is dual to that obtained by Mrozek (\cite{mro1}), in a very general context, for the cohomological Conley index.

\subsection{} To provide an interpretation of $\Gamma$ in dynamical terms fix filtration pairs $(N_R,L_R)$ and $(N_A,L_A)$ for the repeller and the attractor respectively. Denote by $F$ the corresponding domain of the local unstable manifold $W^u_{\rm loc}(N_R)$. (Recall that this is defined as $F := W^u_{\rm loc}(N_R) \cap L_R$). Let $\tau : H_*(F) \longrightarrow H_*(F,F \setminus S)$ be the map induced by the inclusion; as already mentioned its geometric effect is to ``localize'' cycles in $F$ around $S$. With these notations we can now state the following:

\begin{theorem} \label{teo:connecting} For any sufficiently small $N_R$ the pair $(F,F \cap S)$ is admissible for $A$, so that there is a well defined limit map $\Omega : H_*(F,F \setminus S) \longrightarrow CH_*(N_A/L_A)$. The connection homomorphism $\Gamma_q$ is, up to isomorphisms induced by shift equivalences, equal to the composition \[\xymatrix{CH_q(N_R/L_R) \ar[r]^-{\partial_F} & H_{q-1}(F) \ar[r]^-{\tau} & H_{q-1}(F,F \backslash S) \ar[r]^-{\Omega} & CH_{q-1}(N_A/L_A)}\]
\end{theorem}

As usual, the sentence ``up to isomorphisms induced by shift equivalences'' should be interpreted as reflecting the choice of ``coordinate systems'' for $R$ and $A$ afforded by $N_R$ and $N_A$. In the Morse equations it is only ${\rm dim}\ \Gamma_q$ what appears anyway, and by the theorem one has $\dim {\rm im}\ \Gamma_q = \dim {\rm im}\ (\Omega \circ \tau \circ \partial_F)$ regardless of those ``coordinate change'' isomorphisms.

According to the intuitive interpretation of the maps $\partial_F$ and $\Omega$ from the preceding sections, $\Gamma_q$ works as follows. Let $z$ be an element in $CH_q(N_R/L_R)$. We can represent it by a relative $q$--cycle in the local unstable manifold of $R$; its boundary $\partial_F(z)$ is a $(q-1)$--cycle in $F$. By regarding $\partial_F(z)$ as an element of $H_{q-1}(F,F \setminus S)$ via the map $\tau$ we trim it down to a very small neighbourhood of $F \cap S$. Then acting on it with the dynamics will eventually carry it into $N_A$; further iteration will eventually produce an element in $CH_*(N_A/L_A)$ which is $\Gamma_{q-1}(z)$. Before proving the theorem we include some examples to illustrate how these ideas allow, in favorable circumstances, to compute $\Gamma$ just by inspection.

\begin{example} The standard result that if there are no connecting orbits from $R$ to $A$ then $\Gamma = 0$ follows inmediately from the theorem: in that situation $F \cap S = \emptyset$ and so $H_*(F,F \setminus S) = 0$. Thus $\tau = 0$ and so $\Gamma = 0$.
\end{example}

\begin{example} Suppose that $Z_0 := F\cap S$ can be decomposed as the disjoint union of finitely many compact sets $Z_{01},\ldots,Z_{0n}$. Let $Z$ be a compact neighbourhood of $F \cap S$ in $F$ constructed by taking the union of disjoint compact neighbourhoods $Z_i$ of each of the pieces $Z_{0i}$. Then by excision \[H_*(F, F \backslash S) = H_*(Z, Z \backslash S) = \oplus_{i=1}^n H_*(Z_i,Z_i \backslash Z_{0i}).\] The maps $\tau$ and $\Omega$ can then be written as sums $\tau = \sum_{i=1}^n \tau_i$ and $\Omega = \sum_{i=1}^n \Omega_i$ where each \[\tau_i : H_*(F) \longrightarrow H_*(Z_i,Z_i \backslash S)\] and each \[\Omega_i : H_*(Z_i,Z_i \backslash S) \longrightarrow C H_*(N_A/L_A)\] is just the $\Omega$-map corresponding to the admissible pair $(Z_i,Z_i \cap S)$. Consequently the connection homomorphism $\Gamma$ is a sum \[\Gamma = \sum_{i=1}^n \Omega_i \circ \tau_i \circ \partial_F.\]

Now suppose the set $S$ is the union of finitely many disjoint compact invariant sets $S_1, \ldots, S_n$ so that $S_i \cap S_j = A \cup R$ for $i \neq j$. Intuitively, the connecting orbits between $A$ and $R$ can be bundled into disjoint packages indexed by $i$. In this setting evidently $F \cap S$ decomposes as the disjoint union of the compact sets $F \cap S_i$ and we have the sum decomposition $\Gamma = \sum_{i=1}^n \Omega_i \circ \tau_i \circ \partial_F$ derived above. However, each $S_i$ is itself an isolated invariant set having the attractor-repeller decomposition $\{A,R\}$ and the connection map of that sequence is exactly $\Omega_i \circ \tau_i \circ \partial_F$. Thus in this setting the connection map for the attractor-repeller decomposition $\{A,R\}$ of $S$ is the sum of the connection maps for the attractor-repeller decomposition $\{A,R\}$ of the $S_i$. This generalizes \cite[Theorem 2.5, p. 199]{mccord2} to discrete, possibly noninvertible, dynamical systems.
\end{example}

For the next example we need to recall a definition. Suppose the phase space is a differentiable manifold and $f$ is a $\mathcal{C}^1$ diffeomorphism. A fixed point $p$ is called hyperbolic if the derivative $D_p f$ has no (complex) eigenvalues of modulus $1$. The number $d$ of eigenvalues of $D_p f$ with modulus greater than one, counted with multiplicity, is called the index of $p$. By the Hartman-Grobman theorem the dynamics near $p$ is conjugate via a homeomorphism to that of $D_p f$ (near $0$), and using this model one can construct a filtration pair $(N,L)$ for $p$ where $N$ is a box and $L$ is a thickening of some of the faces of the box; namely those transverse to the unstable direction. (Figure \ref{fig:receiver1}.(b) is an example in the case $d = 1$). For such a filtration pair $W^u_{\rm loc}$ is a $d$--dimensional disk and $F$ is a thickening of its boundary. The quotient $W^u_{\rm loc}/F$ is (up to homotopy) a $d$--sphere, and $f_{N/L}$ acts on it by pushing points away from $p$ and moving them towards $[F]$. We can work with $\mathbb{Z}$ coefficients and the Conley index $CH_*(N/L)$ is nonzero only in dimension $d$, where it is $(\mathbb{Z},\pm {\rm Id})$. (For a more careful derivation via filtration pairs see \cite[Theorem 3.1, p. 156]{mro2}).

Consider an attractor-repeller decomposition where the repeller $R$ is a hyperbolic fixed point. For ease of visualization let phase space be a $3$--manifold and $R$ have index $2$. The local situation around $R$ is (up to homeomorphism, again by the Hartman-Grobman theorem) as depicted in Figure \ref{fig:example}.(a). The drawing shows a dotted $2$--disk $g_R$ which represents a generator $z_R$ of $CH_2(R)$. The only possibly nonzero connection homomorphism is $\Gamma_2 : CH_2(R) \longrightarrow CH_1(A)$, and to describe it we need to compute $\Gamma_2(z_R) = \Omega \circ \tau \circ \partial_F(z_R)$. Since $g_R$ is already contained in the local unstable manifold of $N_R$ the map $\partial_F(z_R)$ simply returns the boundary of the disk; i.e. (a fundamental class of) the $1$--sphere $\partial g_R$ shown as a jagged curve in the drawing. Thus to compute $\Gamma_2(z_R)$ we only need to take a sufficiently small neighbourhood of $\partial g_R \cap S$ and act on it with $\Omega$. This is illustrated in the following example:

\begin{example} Suppose that $A$ is also a hyperbolic fixed point with index $1$ (if it has a different index then $\Gamma_2 = 0$ automatically). The local picture near $A$ is as in Figure \ref{fig:receiver1}.(a). Assume, as in the discussion by McCord in \cite[Section 3, p. 200ff.]{mccord2}, that $W^s(A)$ intersects $W^u(R)$ transversally. Then (by dimension counting) this intersection consists only of finitely many arcs $\gamma_i$ joining $R$ to $A$, and $S$ is the union of the closures $\overline{\gamma}_i = R \cup \gamma_i \cup A$ of some of these arcs, say $\overline{\gamma}_1,\ldots,\overline{\gamma}_n$. Here $\partial g_R \cap S$ consists just of the $n$ points where the $\gamma_i$ intersect $\partial g_R$, and so $\tau \circ \partial_F(z_R)$ consists of $n$ short arcs which are transverse to the stable manifold of $A$ since they are contained in the unstable manifold of $R$. Figure \ref{fig:example}.(b) illustrates this for $\gamma_1$. When any one of these short arcs evolves in time and approaches $A$ it looks like the arc $c'$ depicted in Figure \ref{fig:receiver3} because it is still transverse to $W^s(A)$, so $\Omega$ acting on it will produce $\pm z_A$, where $z_A$ is a generator of $CH_1(A)$. Thus $\Gamma_2(z_R) = \left( \sum_{i=1}^n \epsilon_i \right) z_A$, where each $\epsilon_i = \pm 1$. If we had a more detailed description of how $S$ sits in $W^s(A)$ we could determine the relative signs of the $\epsilon_i$ as follows. The dotted $2$--chain $g_R$ which represents $z_R$ is in fact oriented and gives $\partial g_R$ an orientation. This is inherited, after the map $\tau$ acts, by the $n$ short arcs into which $\tau$ splits $\partial g_R$. When the arcs propagate under the dynamics and arrive in $N_A$ this orientation determines whether they are homologous to $z_A$ or $-z_A$, ultimately depending on whether there is some amount of twisting of $W^s(A)$ around $S$.    
\end{example}

The simplicity of these computations stems from the decoupling of $\Gamma$ into the local behaviour near $A$ and $R$ (captured by $\partial$ and $\Omega$ and well understood by linearization) and the geometric map $\tau$. In fact, one can easily adapt the analysis to other situations where $W^u(R)$ and $W^s(A)$ do not intersect transversally since that only changes the map $\tau$ but not $\partial$ or $\Omega$:

\begin{example} \label{ex:notrans} Returning to Figure \ref{fig:example}.(b), imagine that the portion of $W^s(A)$ depicted there was actually tangent to $W^u(R)$ along $\gamma_1$ and laid locally entirely above $F$. Then when the short arc of $\tau \circ \partial_F(z_R)$  localized around $\gamma_1$ approaches $A$ (upon iteration with the dynamics) it is tangent to $W^s(A)$ and lies entirely to one side of it so the map $\Omega$ sends it to zero. Allowing these more general patterns of intersection we still have that $\Gamma_2$ is multiplication by $\sum_{i=1}^n \epsilon_i$ but some of the $\epsilon_i$ may be zero.
\end{example}

\begin{figure}[h!]
\null\hfill
\subfigure[]{
\begin{pspicture}(0,0)(5,7)
	\rput[bl](0,0){\scalebox{0.8}{\includegraphics{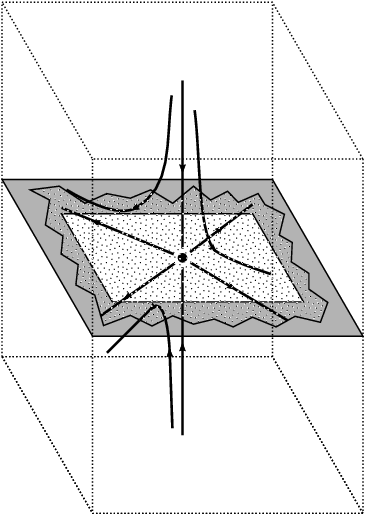}}}
	\rput[bl](0,0.5){$N_R$}
	\rput[tl](4.4,2.2){$F$}
\end{pspicture}}
\hfill
\subfigure[]{
\begin{pspicture}(0,0)(6.5,7)
	\rput[bl](0,1){\scalebox{0.8}{\includegraphics{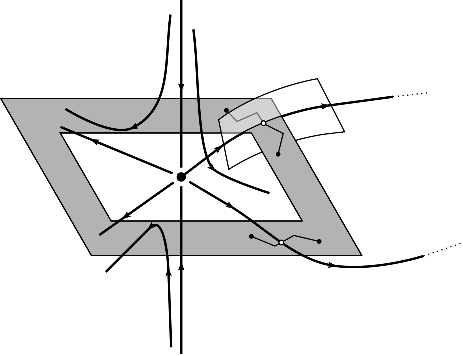}}}
	\rput(4.4,5){$W^s(A)$}
        \rput[bl](5.2,4){$\gamma_1$}
        \rput[bl](5.2,1.8){$\gamma_i$}
\end{pspicture}}
\hfill\null
\caption{ \label{fig:example}}
\end{figure}

All these computations obviously work the same in any dimension (and for any index $d$) and generalize the two corollaries of \cite[Theorem 3.1, p. 200]{mccord2} to discrete dynamics and also situations when $W^s(A)$ and $W^u(R)$ are not transverse. The theorem itself can be recovered as a particular case: if $S$ consists of just one connecting orbit (i.e. $n = 1$) then $\Gamma_d$ is an isomorphism (multiplication by $\epsilon = \pm 1$) and it follows from the exact sequence of the attractor-repeller decomposition that $CH_*(S) = 0$.

\subsection{Proof of Theorem \ref{teo:connecting}} As in the derivation of the exact sequence of the attractor-repeller at the beginning of this section, let $(N_0,N_1,N_2)$ be a filtration adapted to the decomposition; i.e. $(N_0,N_2)$, $(N_0,N_1)$ and $(N_1,N_2)$ are filtration pairs for $S$, $R$ and $A$ respectively. We prove the theorem in three cases of increasing generality.
\smallskip

{\it Case 1.} Consider the particular case when $(N_R,L_R) = (N_0,N_1)$ and $(N_A,L_A) = (N_1,N_2)$. Denote by $W^u_{\rm loc}(N_0)$ the local unstable manifold of $N_0$, so that $F$ is its intersection with $N_1$.
\medskip

{\it Claim.} Consider the filtration pair $(N_1,N_2)$ for $A$. Then $({F},{F} \cap S)$ satisfies the admissibility conditions with $n = 0$ and $U = F$.

{\it Proof of claim.} First, since ${F} \cap S$ is a subset of $S$, we have $f^n({F} \cap S) \subseteq S \subseteq N_0 \setminus N_2$ for every $n \geq 0$. But ${F} \cap S$ is a subset of $N_1$ by definition and $N_1$ is positively invariant in $N_0$, so $f^n({F} \cap S) \subseteq N_1 \setminus N_2$ for every $n \geq 0$. Thus ${F} \cap S \subseteq W^s_{\rm loc}(N_1)$. To show that ${F} \backslash ({F} \cap S) \subseteq N_1 \backslash W^s_{\rm loc}(N_1)$ we only need to check that if $x \in F$ belongs to $W^s_{\rm loc}(N_1)$ then it belongs to $S$. This is indeed true: $x$ has a negative semiorbit $(x_n)_{n \geq 0}$ contained in $N_0$ because $F \subseteq W^u_{\rm loc}(N_0)$, and its forward semiorbit $(x_n)_{n \geq 0}$ is contained in $N_1 \setminus N_2$. Since $N_2$ is positively invariant in $N_0$, the negative semiorbit $(x_n)_{n \geq 0}$ must actually be contained in $N_0 \setminus N_2$ since otherwise the positive semiorbit would be contained in $N_2$. Thus $x \in {\rm Inv}\ \overline{N_0 \setminus N_2} = S$.
\medskip

Since for the particular filtration pair $(N_1,N_2)$ the admissibility conditions are satisfied with $n = 0$ and $U = F$, we may express the map $\Omega_{N_1/N_2}$ as \[\xymatrix{H_*(F,F \backslash S) \ar[r]^-{\pi_*} & H_*(N_1/N_2,N_1/N_2 \backslash W^s_{\rm loc}(N_1)) \ar[r]^-{\omega_{N_1/N_2}} & CH_*(N_1/N_2)}\] where $\pi$ accounts for the inclusion of $(F, F \setminus S)$ into $(N_1,N_1 \backslash W^s_{\rm loc}(N_1))$ followed by the quotient map.

Consider the triples $(W^u_{\rm loc}(N_0),{F},\emptyset) \subseteq (N_0,N_1,N_2)$ and the following portion of their long sequences: \[\xymatrix{H_q(N_0,N_1) \ar[r]^-{\partial} & H_{q-1}(N_1,N_2) \\ H_q(W^u_{\rm loc}(N_0),F) \ar[r]^-{\partial} \ar[u] & H_{q-1}(F) \ar[u]}\] where the two vertical arrows are induced by inclusions. If we now replace each of the three groups that involve pairs of spaces with a quotient (appealing to the strong excision property of \v{C}ech homology) the diagram turns into \[\xymatrix{H_q(N_0/N_1,[N_1]) \ar[r]^-{\Gamma_q} & H_{q-1}(N_1/N_2,[N_2]) \\ H_q(W^u_{\rm loc}(N_0)/F,[F]) \ar[u]^{j_*} \ar[r]^-{\Delta} & H_{q-1}({F}) \ar[u]}\] Now the left vertical arrow is induced by the inclusion $j : W^u_{\rm loc}/F \subseteq N_0/N_1$ and the right vertical arrow is induced by the inclusion of $F$ in $N_1$ followed by the canonical projection $N_1 \longrightarrow N_1/N_2$.


We enlarge the above diagram on the right: \[\xymatrix{ H_q(N_0/N_1,[N_1]) \ar[r]^-{\Gamma_q} & H_{q-1}(N_1/N_2,[N_2]) \ar[r]^-{i_*} &  H_{q-1}(N_1/N_2, N_1/N_2 \backslash W^s_{\rm loc}(N_1)) \\  H_q(W^u_{\rm loc}(N_0)/{F},[{F}]) \ar[u]^{j_*} \ar[r]^-{\Delta} & H_{q-1}({F}) \ar[u] \ar[r]^-{{\tau}} & H_{q-1}({F},{F} \backslash ({F} \cap S)) \ar[u]^{\pi_*} }\] The rightmost vertical arrow is the map $\pi$ induced by the inclusion $(F, F \backslash (F \cap S)) \subseteq (N_1,N_1 \backslash W^s_{\rm loc}(N_1))$ followed by the projection onto $N_1/N_2$ and so the right square commutes because it does so at the point set level already.

Take an element $z \in CH_q(N_0/N_1)$. By Proposition \ref{prop:algebraic_rep} there exists a unique algebraic element $w \in H_q(W^u_{\rm loc}(N_0)/F,[F])$ such that $j_*(w) = z$; recall that by definition $\partial_{F}(z) = \Delta(w)$. The commutativity of the diagram implies that \[i_*(\Gamma_q(z)) = \pi_* \circ \tau \circ \Delta(w) 
= \pi_* \circ \tau \circ \partial_{F}(z).\] Apply the map $\omega_{N_1/N_2}$ on both sides. Since $\Gamma_q$ is equivariant we have that $\Gamma_q(z) \in CH_{q-1}(N_1/N_2)$ and so $\omega_{N_1/N_2} \circ i_* = \omega$ leaves it unchanged. Thus \[\Gamma_q(z) = \omega_{N_1/N_2} \circ \pi_* \circ \tau \circ \partial_{F}(z).\] Finally, in this particular case $\Omega_{N_1/N_2} = \omega_{N_1/N_2} \circ \pi_*$ so we have $\Gamma_q = \Omega_{N_1/N_2} \circ \tau \circ \delta_F$.

\medskip

{\it Case 2.} Assume now that $N_R \subseteq N_0$ and $L_R$ is arbitrary. We still take $(N_A,L_A) = (N_1,N_2)$. Modify the pair $(N_0,N_1)$ as follows. For every $k \geq 0$ define $N_1^{(k)}$ to be the closure of the interior of \[\{x \in N_0 : f^k_{N_0/N_1}([x]) = [N_1]\}.\] Then $(N_0,N_1^{(k)})$ is a filtration pair for $R$ (this is standard; see for instance the proof of \cite[Theorem 4.3, p. 3311]{franksricheson1}). Its local unstable manifold is the same as the local unstable manifold of $N_0$, and we define the corresponding domains $F^{(k)} := W^u_{\rm loc}(N_0) \cap N_1^{(k)}$. Observe that these sets get bigger with increasing $k$ (as do the $N_1^{(k)}$). The same argument as in Case 1 shows that each pair $(F^{(k)},F^{(k)} \cap S)$ is admissible for $A$ (with $n = k$ and $U = F^{(k)}$ in the definition of admissibility), and so there are limit maps $\Omega^{(k)} : H_*(F^{(k)},F^{(k)} \setminus S) \longrightarrow CH_*(N_1/N_2)$. 
\smallskip

{\it Claim.} There exists $k$ such that $L_R \subseteq N_1^{(k)}$. In particular $(F,F \cap S)$ is admissible for $A$ since it is contained in the admissible pair $(F^{(k)},F^{(k)} \cap S)$.

{\it Proof of claim.} First observe that the forward orbit of any $x \in L_R$ must hit $N_1$ before it leaves $N_0$ (if it ever does). Indeed: If $x \not\in S$ then its forward orbit must exit $N_0$ through $N_2$, so it must first hit $N_1$; if $x \in S$ then its forward orbit never leaves $N_0$ and, since $x \not\in R$, converges to $A$ and so must eventually enter $N_1$ since the latter is a neighbourhood of $A$ in $N_0$. This shows that $L_R \subseteq \bigcup_{k \geq 0} N_1^{(k)}$, and (for the same reason as in the argument preceding the proof of Proposition \ref{prop:can_desc_omega}) in fact $L_R \subseteq \bigcup_{k \geq 0} {\rm int}\ N_1^{(k)}$. Since $L_R$ is compact, there exists $k$ such that $L_R \subseteq N_1^{(k)}$. $_{\blacksquare}$
\smallskip

Now we have $(N_R,L_R) \subseteq (N_0,N_1^{(k)}) \supseteq (N_0,N_1)$. Let $r : N_R/L_R \longrightarrow N_0/N_1^{(k)}$ and $r' : N_0/N_1 \longrightarrow N_0/N_1^{(k)}$ be the natural shift equivalences of Example \ref{ex:inclusion}. There is a diagram  \[\xymatrix{ CH_q(N_0/N_1) \ar[d]_{r'_*} \ar[r]^{\partial_{F^{(0)}}} & H_{q-1}(F^{(0)}) \ar[d] \\ CH_q(N_0/N_1^{(k)}) \ar[r]^{\partial_{F^{(k)}}}  & H_{q-1}(F^{(k)}) \\ CH_q(N_R/L_R) \ar[u]^{r_*} \ar[r]^{\partial_F} & H_{q-1}(F) \ar[u]}\] which commutes by Lemma \ref{lem:include_bdry}. This diagram can be enlarged on the right with the respective $\tau$ and $\Omega$ maps: \[\xymatrix{ CH_q(N_0/N_1) \ar[d]_{r'_*}^{\cong} \ar[r]^{\partial_{F^{(0)}}} & H_{q-1}(F^{0}) \ar[d] \ar[r]^-{\tau^{(0)}} & H_{q-1}(F^{(0)},F^{(0)} \backslash S) \ar[d] \ar[dr]^{\Omega^{(0)}} & \\ CH_q(N_0/N_1^{(k)}) \ar[r]^{\partial_{F^{(k)}}}  & H_{q-1}(F^{(k)}) \ar[r]^-{\tau^{(k)}} & H_{q-1}(F^{(k)}, F^{(k)} \backslash S) \ar[r]^{\Omega^{(k)}} & CH_{q-1}(N_1/N_2) \\ CH_q(N_R/L_R) \ar[u]^{r_*}_{\cong} \ar[r]^{\partial_F} & H_{q-1}(F) \ar[u] \ar[r]^-{\tau} & H_{q-1}(F,F \setminus S) \ar[u] \ar[ur]^{\Omega} & }\] This is again commutative because of the remark in p. \pageref{pg:remark} following the proof of Theorem \ref{teo:indepOmegaNL}. The maps $r_*$ and $r'_*$ restricts to isomorphisms between the Conley indices by Proposition \ref{prop:W2invariant}. From Case 1 we know that the connection homomorphism $\Gamma_q : CH_q(N_0/N_1) \longrightarrow CH_{q-1}(N_1/N_2)$ is the composition $\Omega^{(0)} \circ \tau^{(0)} \circ \partial_{F^{(0)}}$. By the commutativity of the diagram we then have \[\Gamma_q = \Omega^{(k)} \circ \tau^{(k)} \circ \partial_{F^{(k)}} \circ r'_* = \Omega \circ \tau \circ \partial_F \circ (r_*^{-1} \circ r'_*).\] Let $s$ be a shift inverse for $r$; it satisfies $s_* \circ r_* = (f_{N_R/L_R})_*^m$ for some integer $m \geq 0$. When restricted to $CH_*(N_R/L_R)$ all maps become isomorphisms, and so $r_*^{-1} = \phi_{N_R/L_R}^{-m} \circ s_*$ is the inverse of $r_*$. Thus \[r_*^{-1} \circ r'_* = \phi_{N_R/L_R}^{-m} \circ (s \circ r')_* = (s \circ r')_* \circ \phi_{N_0/L_0}^{-m}\] and so \[\Gamma_q \circ \phi_{N_0/L_0}^m = (\Omega \circ \tau \circ \partial_F) \circ (s \circ r')_*\] which shows that $\Gamma_q$ equals the composition $\Omega \circ \tau \circ \partial_F$ save for the isomorphisms induced by the shift equivalences $f_{N_0/N_1}^m : N_0/N_1 \longrightarrow N_0/N_1$ (a translation of the origin of times) and $s \circ r' : N_0/N_1 \longrightarrow N_R/L_R$.
\medskip

{\it Case 3.} This is the end of the proof. One can remove the condition $(N_A,L_A) = (N_1,N_2)$ simply by using Theorem \ref{teo:indepOmegaNL}. This only adds a shift in the origin of times and an isomorphism induced by a shift equivalence. The theorem is then proved for any choice of $(N_R,L_R)$ and $(N_A,L_A)$ as long as $N_R \subseteq N_0$, which is the ``sufficiently small'' condition in the statement of the theorem.

\bibliographystyle{plain}

\bibliography{biblio}

\end{document}